\documentclass[reqno,english]{amsart}%
\usepackage{amsfonts,amsmath,latexsym,verbatim,amscd,mathrsfs,color,array}
\usepackage[colorlinks=true]{hyperref}
\usepackage{amsmath,amssymb,amsthm,amsfonts,graphicx,color}
\usepackage{amssymb}
\usepackage{pdfsync}
\usepackage{epstopdf}
\usepackage{cite}
\usepackage{graphicx}
\usepackage{amsmath}
\usepackage{amsfonts}%
\providecommand{\U}[1]{\protect\rule{.1in}{.1in}}

\numberwithin{equation}{section}

\newtheorem{theorem}{Theorem}[section]
\newtheorem{corollary}{Corollary}[section]
\newtheorem{lemma}{Lemma}[section]
\newtheorem{proposition}{Proposition}[section]
\newtheorem{remark}{Remark}[section]
\newtheorem{definition}{Definition}[section]

\numberwithin{equation}{section}

\newcommand{\bbr}{\mathbb{R}}

\newcommand{\ve}{\varepsilon}

\newcommand{\bd}{\begin{definition}}
\newcommand{\ed}{\end{definition}}

\newcommand{\br}{\begin{remark}}
\newcommand{\er}{\end{remark}}

\newcommand{\be}{\begin{equation}}
\newcommand{\ee}{\end{equation}}

\newcommand{\bc}{\begin{corollary}}
\newcommand{\ec}{\end{corollary}}

\begin{document}

\title[Extremal values of $L^2$-Pohozaev manifolds]{Extremal values of $L^2$-Pohozaev manifolds and their applications}

\author[T. Liu]{Taicheng Liu}
\address{\noindent School of Mathematics, China
University of Mining and Technology, Xuzhou, 221116, P.R. China}
\email{2721771800@qq.com}

\author[Y.Wu]{Yuanze Wu}
\address{\noindent  School of Mathematics, China
University of Mining and Technology, Xuzhou, 221116, P.R. China}
\email{wuyz850306@cumt.edu.cn}

\begin{abstract}
In this paper, we consider the following Schr\"{o}dinger equation:
\begin{equation*}
\begin{cases}
-\Delta u=\lambda u+\mu|u|^{q-2}u+|u|^{2^*-2}u\quad\text{in }\mathbb{R}^N,\\
\int_{\mathbb{R}^N}|u(x)|^2dx=a,\quad u\in H^1(\mathbb{R}^N),\\
\end{cases}
\end{equation*}
where $N\ge 3$, $2<q<2+\frac{4}{N}$, $a, \mu>0$, $2^*=\frac{2N}{N-2}$ is the critical Sobolev exponent and $\lambda\in \mathbb{R}$ is one of the unknowns in the above equation which appears as a Lagrange multiplier.  By applying the minimization method on the $L^2$-Pohozaev manifold, we prove that if $N\geq3$, $q\in\left(2,2+\frac{4}{N}\right)$, $a>0$ and $0<\mu\leq\mu^{*}_{a}$, then the above equation has two positive solutions which are real valued, radially symmetric and radially decreasing,
where
\begin{equation*}
\mu^*_a=\frac{(2^*-2)(2-q\gamma_q)^{\frac{2-q\gamma_q}{2^*-2}}}{\gamma_q(2^*-q\gamma_q)^{\frac{2^*-q\gamma_q}{2^*-2}}}\inf_{u\in H^1(\mathbb{R}^N), \|u\|_{2}^2=a}\frac{\left(\|\nabla u\|_2^2\right)^\frac{2^*-q\gamma_q}{2^*-2}}{\|u\|_q^q\left(\|u\|_{2^*}^{2^*}\right)^{\frac{2-q\gamma_q}{2^*-2}}}.
\end{equation*}
Our results improve the conclusions of \cite{JeanjeanLe2021,JeanjeanJendrejLeVisciglia2022,Soave2020-2,WeiWu2022} and we hope that our proofs and discussions in this paper could provide new techniques and lights to understand the structure of the set of positive solutions of the above equations.

\vspace{3mm} \noindent{\bf Keywords:} Schr\"odinger equation, Sobolev exponent, Variational method, Pohozaev manifold, Extremal value.

\vspace{3mm}\noindent {\bf AMS} Subject Classification 2010: 35B09; 35B33; 35J20.%

\end{abstract}

\date{}
\maketitle

\section{Introduction}
The Schr\"{o}dinger equation is one of the most famous models in mathematics and physics.  The typical nonlinear model reads as
\begin{eqnarray}\label{eqnWu1001}
\iota\psi_t+\Delta \psi+g\left(\left|\psi\right|^2\right)\psi=0\quad\text{in }\mathbb{R}^N,
\end{eqnarray}
where $\iota$ is the imaginary unit, $\psi$ is the macroscopic wave function and $g\left(\left|\psi\right|^2\right)$ is the nonlinear potential.  This typical model~\eqref{eqnWu1001} is widely used in many physical fields such as nonlinear optics, plasma, Bose-Einstein condensates, superconducting mean fields and so on.  In many applications, the natural choice of the nonlinear potential is that $g\left(\left|\psi\right|^2\right)=\left|\psi\right|^2$, which describes the attractive interplay of atoms in medium density or laser beams in some kinds of materials.  Sometimes, to describe  the physical phenomenon more precisely, the nonlinear potential will be chosen to be the cubic-quintic type (cf. \cite{KillipOhPocovnicuVisan2017}), that is,  $g\left(\left|\psi\right|^2\right)=\left|\psi\right|^2-\left|\psi\right|^4$, which is a special form of the mathematical generalization $g\left(\left|\psi\right|^2\right)=\mu\left|\psi\right|^{q}+\nu\left|\psi\right|^{p}$ (called combined or mixed nonlinearities in mathematical references if $q\not=p$), where $q,p>1$ and $\mu,\nu\in\bbr$.  The studies on the Schr\"{o}dinger equation~\eqref{eqnWu1001} with combined nonlinearities were initialed by Tao-Visan-Zhang in \cite{TaoVisanZhang2007}, where local well posedness of the unique strong solution is established.  Moreover, it has also known that the unique strong solution also has conservations of energy and mass.  Some further related studies on  the Schr\"{o}dinger equation~\eqref{eqnWu1001} with combined nonlinearities can be found in \cite{AkahoriIbrahimKikuchiNawa2021,ChengMiaoZhao2016,KillipOhPocovnicuVisan2017,MiaoXuZhao2013,MiaoZhaoZheng2017} and the references therein.  

\vskip0.12in

In understanding the global dynamics (cf. stability and instability), the standing waves of the Schr\"{o}dinger equation~\eqref{eqnWu1001}, which are of the form $\psi(t,x)=e^{i\lambda t}u(x)$ where $\lambda\in\bbr$ is a constant, are crucial.  The time-independent function $u(x)$ in the standing waves satisfies the following elliptic equation:
\begin{eqnarray}\label{eqnWu1002}
\left\{\aligned  -\Delta u&=\lambda u+f(u)\quad\text{in }\mathbb{R}^N,\\
\int_{\mathbb{R}^N}u^2&=a,\quad u\in H^1(\mathbb{R}^N),\endaligned\right.
\end{eqnarray} 
where $f(u)=g\left(|u|^2\right)u$ and $a$ is the mass of the initial data according to the conservation of mass of the Schr\"odinger flow.  The studies on \eqref{eqnWu1002}, started from the semi-classical papers \cite{BerestyckiLions1983,BerestyckiLions1983-1,Lions1984,Lions1984-1,Strauss1977}, have a long history.  In particular, in recent papers \cite{Soave2020-1,Soave2020-2}, Soave initialed the studies on the special form of \eqref{eqnWu1002} with combined nonlinearities
\begin{equation}\label{eqnWu0008}
\left\{\aligned -\Delta u=\lambda u&+\mu|u|^{q-2}u+|u|^{p-2}u, x\in \mathbb{R}^N,\\
\int_{\mathbb{R}^N}u^2&=a,\quad u\in H^1(\mathbb{R}^N)\endaligned\right.
\end{equation}
where $2<q<p\leq2^*$ and $\mu>0$.  The special case of $p=2^*$ reads as
\begin{equation}\label{eqd106}
\left\{\aligned -\Delta u&=\lambda u+\mu|u|^{q-2}u+|u|^{2^*-2}u, x\in \mathbb{R}^N,\\
\int_{\mathbb{R}^N}u^2&=a,\quad u\in H^1(\mathbb{R}^N).\endaligned\right.
\end{equation}
To precisely state Soave's results in \cite{Soave2020-2}, we need to introduce some necessary notations and definitions.  The corresponding functional of \eqref{eqd106} is given by 
\begin{equation*}
\Psi_{\mu}(u)=\frac{1}{2}\|\nabla u\|_2^2-\frac{\mu}{q}\|u\|_q^q-\frac{1}{2^*}\|u\|_{2^*}^{2^*},
\end{equation*}
where $\|\cdot\|_{p}$ is the usual norm in the Lebesgue space $L^p(\bbr^N)$.
The $L^2$-Pohozaev manifold of $\Psi_{\mu}(u)$, which is introduced by Bartsch and Soave in \cite{BartschSoave2017}, is defined by
\begin{eqnarray*}
\mathcal{P}_{a,\mu}=\{u\in\mathcal{S}_a\mid \|\nabla u\|_2^2=\mu \gamma_q\|u\|_q^q+\|u\|_{2^*}^{2^*}\},
\end{eqnarray*}
where $\gamma_q=\frac{N(q-2)}{2q}$ and
\begin{equation*}
\mathcal{S}_a=\left\{u\in H^1(\mathbb{R}^N)\mid\|u\|_2^2=a\right\}.
\end{equation*}
We say that $\hat{u}$ is a ground-state solution of the equation~\eqref{eqd106}, if $\hat{u}$ solves \eqref{eqd106} for some suitable $\lambda\in\bbr$ and $\Psi_\mu(\hat{u})=\inf_{u\in \mathcal{P}_{a, \mu}}\Psi_\mu(u)$.
Now, one of Soave's results in the recent paper \cite{Soave2020-2} can be stated as follows.
\begin{theorem}\label{theo1}
Let $N\ge 3$ and $2<q<2+\frac{4}{N}$.  Then there exists $\alpha_{N,q}=\min\{C_1,C_2\}$, with 
\begin{align*}
C_1&=\left(\frac{2^*S^{\frac{2^*}{2}}(2-q\gamma_q)}{2(2^*-q\gamma_q)}\right)^{\frac{2-q\gamma_q}{2^*-2}}\frac{q(2^*-2)}{2C_{N,q}^q(2^*-q\gamma_q)},\\
C_2&=\frac{22^*}{N\gamma_qC_{N,q}^q(2^*-q\gamma_q)}\left(\frac{q\gamma_qS^{\frac{N}{2}}}{2-q\gamma_q}\right)^{\frac{2-q\gamma_q}{2}}
\end{align*}
and $C_{N,q}$ and $S$ the optimal constants of the Gagliardo-Nirenberg and Sobolev inequalities, such that if $0<\mu a^{\frac{q-q\gamma_q}{2}}<\alpha_{N,q}$ then the equation~\eqref{eqd106} has a ground-state solution $\hat{u}_{a,\mu}$.
\end{theorem}

\vskip0.12in

In the $L^2$-subcritical case $2<q<2+\frac{4}{N}$, since $\Psi_\mu(u)|_{\mathcal{S}_a}$ is unbounded from below, it could be naturally to expect that $\Psi_\mu(u)|_{\mathcal{S}_a}$ has a second critical point of the mountain-pass type, which is also positive, real valued and radially symmetric.  This natural expectation has been pointed out by Soave in \cite[Remark~1.1]{Soave2020-2} which can be summarized to be the following question:
\begin{enumerate}
\item[$(Q_1)$]\quad {\bf Does $\Psi_\mu(u)|_{\mathcal{S}_a}$ have a critical point of the mountain-pass type in the $L^2$-subcritical case $2<q<2+\frac{4}{N}$?}
\end{enumerate}
The question~$(Q_1)$ has been explored by Jeanjean and Le in \cite{JeanjeanLe2021} and Wei and Wu in \cite{WeiWu2022}, independently.  To state the improvements of Jeanjean-Le and Wei-Wu precisely, we need to introduce more notations.  By the fibering maps
\begin{eqnarray*}
\Phi_{\mu,u}(s)=\frac{e^{2s}}{2}\|\nabla u\|_2^2-\frac{\mu e^{q\gamma_qs}}{q}\|u\|_q^q-\frac{e^{2^*s}}{2^*}\|u\|_{2^*}^{2^*},
\end{eqnarray*}
which is introduced by Jeanjean in \cite{Jeanjean1997}, $\mathcal{P}_{a,\mu}$ can be naturally divided into the following three parts:
\begin{eqnarray*}
\mathcal{P}^+_{a,\mu}=\{u\in\mathcal{P}_{a,\mu} \mid 2\|\nabla u\|_2^2>\mu q\gamma^2_q\|u\|_q^q+2^*\|u\|_{2^*}^{2^*}\},\\
\mathcal{P}^0_{a,\mu}=\{u\in\mathcal{P}_{a,\mu} \mid 2\|\nabla u\|_2^2=\mu q\gamma^2_q\|u\|_q^q+2^*\|u\|_{2^*}^{2^*}\},\\
\mathcal{P}^-_{a,\mu}=\{u\in\mathcal{P}_{a,\mu} \mid 2\|\nabla u\|_2^2<\mu q\gamma^2_q\|u\|_q^q+2^*\|u\|_{2^*}^{2^*}\}.
\end{eqnarray*}
Let
\begin{eqnarray}\label{eqnnWu0099}
m^\pm(a,\mu)=\inf_{u\in \mathcal{P}_{a, \mu}^\pm}\Psi_\mu(u).
\end{eqnarray}
Then we say that $\hat{u}$ is a mountain-pass solution of the equation~\eqref{eqd106}, if $\hat{u}$ solves \eqref{eqd106} for some suitable $\lambda\in\bbr$ and $\Psi_\mu(\hat{u})=m^-(a,\mu)$.  Now, we can summarize the main results in \cite{JeanjeanLe2021,WeiWu2022} into the following theorem.
\begin{theorem}\label{theo2}
Let $N\ge 3$, $2<q<2+\frac{4}{N}$ and $a>0$.  If $0<\mu a^{\frac{q-q\gamma_q}{2}}<\alpha_{N,q}$ then the equation~\eqref{eqd106} has two solutions $u_{a,\mu,\pm}\in H^1(\bbr^N)$ which are real valued, positive, radially symmetric and radially decreasing.  Moreover, $u_{a,\mu,+}$ is a ground-state solution and $u_{a,\mu,-}$ is a mountain-pass solution.
\end{theorem}

\vskip0.12in

Besides, since Soave only considered the case that $\mu a^{q-q\gamma_q}>0$ small in \cite[Theorem~1.1]{Soave2020-2} in the $L^2$-subcritical case $2<q<2+\frac{4}{N}$, it is also natural to ask what will happen if $\mu>0$ and $\mu a^{q-q\gamma_q}>0$ is large in the $L^2$-subcritical case $2<q<2+\frac{4}{N}$.  This natural question has been proposed by Soave in \cite{Soave2020-2} as an open problem, which can be summarized to be the following one:
\begin{enumerate}
\item[$(Q_2)$]\quad {\bf Does $\Psi_\mu(u)|_{\mathcal{S}_a}$ have a ground state if $\mu>0$ and $\mu a^{q-q\gamma_q}>0$ large in the $L^2$-subcritical case $2<q<2+\frac{4}{N}$?}
\end{enumerate}
The question~$(Q_2)$ is explored by Wei and Wu in \cite{WeiWu2023}.  By observing that finding positive solutions of the equation~\eqref{eqd106} is equivalent to finding solutions of the equation
\begin{eqnarray}\label{eqnWu1020}
t^{\frac{2}{q\gamma_q-q}-1}-\frac{1-\gamma_q}{a\mu^{\frac{2}{q-q\gamma_q}}}\|v_t\|_q^q=0
\end{eqnarray}
via the Pohozaev identity, where $v_t$ is a positive solution of the following equation
\begin{eqnarray}\label{eqnnWu1020}
\left\{\aligned
&-\Delta v+v=t |v|^{q-2}v+|v|^{2^*-2}v\quad\text{in }\mathbb{R}^N,\\
&v\in H^1(\bbr^N),
\endaligned\right.
\end{eqnarray}
Wei and Wu proved the following theorem in \cite{WeiWu2023}.
\begin{theorem}\label{theo4}
Let $N\ge 3$, $2<q<2+\frac{4}{N}$ and $a>0$.  Then there exists $\hat{\mu}_a\geq a^{-\frac{q-q\gamma_q}{2}}\alpha_{N,q}$ such that the equation~\eqref{eqd106} has no positive solutions for $\mu>\hat{\mu}_a$.
\end{theorem}

\begin{remark}
It is worth pointing out that besides the questions~$(Q_1)$ and $(Q_2)$, there are several other open questions proposed in \cite{Soave2020-1,Soave2020-2} and the recent developments in these open questions can be found in \cite{JeanjeanZhangZhong2024,LiZou2024, QiZou2023} and the references therein.
\end{remark}

\vskip0.12in

Since the functional $\Psi_{\mu}(u)$ has a concave-convex structure on the manifold $\mathcal{S}_a$, by the well-known results of the elliptic equations with concave-convex nonlinearities in \cite{AmbrosettiBrezisCerami1994} and Theorems~\ref{theo2} and \ref{theo4}, it is natural to {\it conjecture} that if $N\ge 3$, $2<q<2+\frac{4}{N}$ and $a>0$ then there exists a unique $\hat{\mu}_{a,*}>0$ such that the equation~\eqref{eqd106} has two positive solutions for $0<\mu<\hat{\mu}_{a,*}$ with the first one being a ground-state solution and the another one being a mountain-pass solution, has one positive solution for $\mu=\hat{\mu}_{a,*}$ and has no positive solutions for $\mu>\hat{\mu}_{a,*}$.  In this paper, we shall provide some evidences to support the above conjecture by improving Theorem~\ref{theo2}.

\vskip0.12in

To state our main results in this paper, we introduce
\begin{equation*}
\mu^*_a=\frac{(2^*-2)(2-q\gamma_q)^{\frac{2-q\gamma_q}{2^*-2}}}{\gamma_q(2^*-q\gamma_q)^{\frac{2^*-q\gamma_q}{2^*-2}}}\inf_{u\in \mathcal{S}_a}\frac{\left(\|\nabla u\|_2^2\right)^\frac{2^*-q\gamma_q}{2^*-2}}{\|u\|_q^q\left(\|u\|_{2^*}^{2^*}\right)^{\frac{2-q\gamma_q}{2^*-2}}}.
\end{equation*}
We call $\mu^*_a$ {\it the extremal value} of the $L^2$-Pohozaev manifold $\mathcal{P}_{a, \mu}$ since 
\begin{eqnarray*}
\mu^*_a=\max\left\{\tau>0\mid \mathcal{P}_{a, \mu}^0=\emptyset\text{ for }0<\mu<\tau\right\},
\end{eqnarray*}
see Proposition~\ref{prop001} for the details.
Now, the our main results can be stated as follows.
\begin{theorem}\label{thm001}
Let $N\ge 3$, $2<q<2+\frac{4}{N}$ and $a>0$.  If $0<\mu\leq\mu^{*}_{a}$ then the equation~\eqref{eqd106} has two solutions $u_{a,\mu,\pm}$ for suitable Lagrange multipliers $\lambda_{a, \mu, \pm}<0$, which are real valued, positive, radially symmetric and radially decreasing.  Moreover, $u_{a,\mu,+}$ is a ground-state solution and $u_{a,\mu,-}$ is a mountain-pass solution.
\end{theorem}

\vskip0.12in

Some remarks about Theorem~\ref{thm001} are in ordered.  As we stated above, Theorem~\ref{thm001} improves Theorem~\ref{theo2}.  Let us denote
\begin{align*}
\mu(u)=\frac{(2^*-2)(2-q\gamma_q)^{\frac{2-q\gamma_q}{2^*-2}}}{\gamma_q(2^*-q\gamma_q)^{\frac{2^*-q\gamma_q}{2^*-2}}}\frac{\left(\|\nabla u\|_2^2\right)^\frac{2^*-q\gamma_q}{2^*-2}}{\|u\|_q^q\left(\|u\|_{2^*}^{2^*}\right)^{\frac{2-q\gamma_q}{2^*-2}}}
\end{align*}
for the sake of simplicity.
Then by the Gagliardo-Nirenberg inequality and the Sobolev inequality, we have
\begin{align*}
\mu(u)\geq\frac{(2^*-2)(2-q\gamma_q)^{\frac{2-q\gamma_q}{2^*-2}}}{\gamma_q(2^*-q\gamma_q)^{\frac{2^*-q\gamma_q}{2^*-2}}}\frac{S^{\frac{2^*(2-q\gamma_q)}{2(2^*-2)}}}{C_{N,q}^qa^{\frac{q-q\gamma_q}{2}}},
\end{align*}
which implies that $\mu(u)a^\frac{q(1-\gamma_q)}{2}\geq C_1\frac{2}{q\gamma_q}\left(\frac{2^*}{2}\right)^{\frac{2-q\gamma_q}{2^*-2}}$, where $C_1$ is given in Theorem~\ref{theo1}.  Since $\frac{2}{q\gamma_q}\left(\frac{2^*}{2}\right)^{\frac{2-q\gamma_q}{2^*-2}}>1$ by $2<q<2+\frac{4}{N}$, we have
\begin{equation*}
\mu^*_aa^\frac{q(1-\gamma_q)}{2}=\inf_{u\in S_a}\mu(u)a^\frac{q(1-\gamma_q)}{2}\geq C_1\frac{2}{q\gamma_q}\left(\frac{2^*}{2}\right)^{\frac{2-q\gamma_q}{2^*-2}}> C_1\ge \alpha_{N,q}.
\end{equation*}

\vskip0.12in

We prove Theorem~\ref{thm001} for $\mu\in(0, \mu^*_{a})$ by studying the variational problems~\eqref{eqnnWu0099}.  The achievement of $m^+(a,\mu)$ can be proved by the usual arguments, see, for example, \cite{JeanjeanLe2021,JeanjeanJendrejLeVisciglia2022,Soave2020-2,WeiWu2022}, since the submanifold $\mathcal{P}_{a,\mu}^+$ consists of local minimum points of the fibering map $\Phi_{\mu,u}(s)$ for all $\mu\in(0, \mu^*_{a})$.  However, the achievement of $m^-(a,\mu)$ can not be dealt with these usual arguments, since on one hand, the key point in these arguments is to use the fact that the submanifold $\mathcal{P}_{a,\mu}^-$ consists of global maximum points of the fibering map $\Phi_{\mu,u}(s)$ to prove the minimizing sequence of $m^-(a,\mu)$ is compact and on the other hand, if $\mu$ close to $\mu^*_{a}$ then $\mathcal{P}_{a,\mu}^-$ only consists of {\it local} maximum points of the fibering map $\Phi_{\mu,u}(s)$.  Thus, to prove the achievement of $m^-(a,\mu)$, we first construct a mountain-pass solution of \eqref{eqnWu0008} in the Sobolev subcritical case $p<2^*$ by proving the following proposition.
\begin{proposition}\label{003}
Let $N\ge 3$, $2<q<2+\frac{4}{N}<p<2^*$, $a>0$ and $0<\mu<\mu^*_{a,p}$, where
\begin{equation}\label{eqnWu0025}
\mu_{a,p}^*=\inf_{u\in \mathcal{S}_a}\frac{(p\gamma_p-2)(2-q\gamma_q)^{\frac{2-q\gamma_q}{p\gamma_p-2}}\left(\|\nabla u\|_2^2\right)^\frac{p\gamma_p-q\gamma_q}{p\gamma_p-2}}{\gamma_q(p\gamma_p-q\gamma_q)^{\frac{p\gamma_p-q\gamma_q}{p\gamma_p-2}}\gamma_p^{\frac{2-q\gamma_q}{p\gamma_p-2}}\|u\|_q^q\left(\|u\|_{p}^{p}\right)^{\frac{2-q\gamma_q}{p\gamma_p-2}}}.
\end{equation}
Then the variational problem
\begin{equation}\label{eqnWu0009}
m^-_p(a,\mu)=\inf_{u\in \mathcal{P}_{a,\mu,p}^-}\Psi_{\mu,p}(u)
\end{equation}
is achieved by some $u_{a,\mu,p,-}$, which is real valued, positive, radially symmetric and radially decreasing, where
\begin{equation*}
\Psi_{\mu,p}(u)=\frac{1}{2}\|\nabla u\|_2^2-\frac{\mu}{q}\|u\|_q^q-\frac{1}{p}\|u\|_{p}^{p}.
\end{equation*}
Moreover, $u_{a,\mu,p,-}$ also satisfies the Schr\"{o}dinger equation~\eqref{eqnWu0008}
for a suitable Lagrange multiplier $\lambda_{a, \mu,p,-}<0$.
\end{proposition}
We next prove the achievement of $m^-(a,\mu)$ by showing that $\mu_{a,p}^*\to\mu_{a}^*$ and $\{u_{a,\mu,p,-}\}$ is a compact minimizing sequence of $m^-(a,\mu)$ as $p\uparrow 2^*$ up to a subsequence.  Our idea is inspired by \cite{QiZou2023} which is also used in \cite{DengWu2023}.  Moreover, Proposition~\ref{003} improves \cite[Theorem~1.3]{Soave2020-1}, see Remark~\ref{rmkWu0002} for the details.

\vskip0.12in

The crucial point in the proof of Theorem~\ref{thm001} for $\mu=\mu_a^*$ is to establish the following energy estimates
\begin{equation*}
m^\pm(a,\mu_a^*)=\inf_{u\in \mathcal{P}_{a, \mu_a^*}^\pm}\Psi_{\mu_a^*}(u)<m^0(a,\mu_a^*)=\inf_{u\in \mathcal{P}_{a, \mu_a^*}^0}\Psi_{\mu_a^*}(u).
\end{equation*}
In \cite{AlbuquerqueSilva2020}, the above energy estimates is obtained by proving the achievement of variational problems which is related to the definitions of extremal values.  In our problem, the involved variational problem is the following
\begin{equation*}
\inf_{u\in \mathcal{S}_a}\frac{\left(\|\nabla u\|_2^2\right)^\frac{2^*-q\gamma_q}{2^*-2}}{\|u\|_q^q\left(\|u\|_{2^*}^{2^*}\right)^{\frac{2-q\gamma_q}{2^*-2}}}.
\end{equation*}
However, this variational problem is quite different from the related one considered in \cite{AlbuquerqueSilva2020} since we loss the compactness of the embeddings from $H^1(\bbr^N)$ into $L^2(\bbr^N)$ and $L^{2^*}(\bbr^N)$ in solving it.  Thus, it is technically quite nontrivial and involved to prove the achievement of this variational problem.  By constructing a suitable test function and using the properties of the digamma function, we are only able to prove that this variational problem is achieved for $N\geq7$.  Thus, to establish the existence results of the Schr\"{o}dinger equation~\eqref{eqd106} at the extremal value $\mu^*_a$ for all $N\geq3$, we introduce a perturbation argument around the degenerate submanifold $\mathcal{P}^0_{a,\mu^*_a}$ to replace the discussions on the achievement of this variational problem.

\vskip0.12in

To provide more evidences in this direction to support the conjecture stated above, it is natural to define
\begin{equation*}
\mu^*_{a,\pm}=\sup\{\mu>0\mid m^\pm(a,\mu)=m^0(a,\mu)\}.
\end{equation*}
Clearly, we have $\mu^*_{a}<\mu^*_{a,\pm}$.
Moreover, by our arguments, one can also prove that $\mu^*_{a,+}\leq\mu^*_{a,-}$ and the equation~\eqref{eqd106} has two positive solutions with one being a ground-state solution and another one being a mountain-pass solution for $0<\mu<\mu^*_{a,+}$.  On the other hand, if $v_t$ is nondegenerate for all $t>0$ or at least only degenerate at an isolated sequence $\{t_n\}$ (By perturbation arguments, it is not hard to prove that $v_t$ is nondegenerate for $t>0$ sufficiently small and sufficiently large), where $v_t$ is the positive solution of \eqref{eqnnWu1020}, then by the continuity method, it is not difficult to find out that if $N\ge 3$, $2<q<2+\frac{4}{N}$ and $a>0$ then there exists $\hat{\mu}_{a,**}>0$ such that the equation~\eqref{eqnWu1020} has two solutions for $0<\mu<\hat{\mu}_{a,**}$, has one solution for $\mu=\hat{\mu}_{a,**}$ and has no solutions for $\mu>\hat{\mu}_{a,**}$.  Combining the above discusions, it is natural to propose the following question:

\vskip0.12in

{\bf Open question:}\quad Let $N\ge 3$, $2<q<2+\frac{4}{N}$ and $a>0$.  Does there hold $\mu^*_{a,+}=\mu^*_{a,-}=\hat{\mu}_{a,**}$?

\vskip0.12in

Clearly, if this question has a positive answer, then the structure of positive solutions of the equation~\eqref{eqd106} will be very similar to that of elliptic equations with concave-convex nonlinearities.

\vskip0.12in

{\bf Notations.}\quad Throughout this paper, $a\sim b$ means that $C'b\leq a\leq Cb$ and $a\lesssim b$ means that $a\leq Cb$ where $C$ and $C'$ are positive constants.  We also denote $\gamma_q=\frac{N(q-2)}{2q}$ for all $2\leq q\leq 2^*$, $(u)_s=s^{\frac{N}{2}}u(sx)$ and
\begin{eqnarray*}
\left\langle\psi,\varphi\right\rangle_{L^2}=\int_{\bbr^N}\psi\varphi dx.
\end{eqnarray*}

\section{Extremal values}
Let us consider the following functional
\begin{equation*}
\Psi_{\mu,p}(u)=\frac{1}{2}\|\nabla u\|_2^2-\frac{\mu}{q}\|u\|_q^q-\frac{1}{p}\|u\|_{p}^{p}
\end{equation*}
where $2<q<2+\frac{4}{N}<p\leq 2^*$.  The related fibering maps are given by
\begin{equation*}
\Phi_{\mu,p,u}(s)=\Psi_{\mu,p}((u)_s)=\frac{s^2}{2}\|\nabla u\|_2^2-\frac{\mu s^{q\gamma_q}}{q}\|u\|_q^q-\frac{s^{p\gamma_p}}{p}\|u\|_{p}^{p}.
\end{equation*}
The $L^2$-Pohozaev manifold of the functional $\Psi_{\mu,p}(u)$, which is given by
\begin{equation*}
\mathcal{P}_{a,\mu,p}=\left\{u\in H^1(\mathbb{R}^N)\cap \mathcal{S}_a~|~\|\nabla u\|_2^2=\gamma_p\|u\|_{p}^{p}+\mu\gamma_q\|u\|_q^q\right\},
\end{equation*}
can be naturally divided into the following three parts by the types of critical points of the fibering maps:
\begin{equation*}
\mathcal{P}_{a,\mu,p}^-=\left\{u\in \mathcal{P}_{a,\mu,p}~|~2\|\nabla u\|_2^2-\mu q\gamma_q^2\|u\|_q^q-p\gamma_p^2\|u\|_{p}^{p}<0\right\},
\end{equation*}
\begin{equation*}
\mathcal{P}_{a,\mu,p}^0=\left\{u\in \mathcal{P}_{a,\mu,p}~|~2\|\nabla u\|_2^2-\mu q\gamma_q^2\|u\|_q^q-p\gamma_p^2\|u\|_{p}^{p}=0\right\},
\end{equation*}
\begin{equation*}
\mathcal{P}_{a,\mu,p}^+=\left\{u\in \mathcal{P}_{a,\mu,p}~|~2\|\nabla u\|_2^2-\mu q\gamma_q^2\|u\|_q^q-p\gamma_p^2\|u\|_{p}^{p}>0\right\}.
\end{equation*}
For any $u\in\mathcal{S}_a$, we define
\begin{equation*}\label{mup}
\left\{\aligned& s_p(u)=\left(\frac{(2-q\gamma_q)\|\nabla u\|_2^2}{(p\gamma_p-q\gamma_q)\gamma_p\|u\|_{p}^{p}}\right)^{\frac{1}{p\gamma_p-2}},\\
&\mu_p(u)=\frac{(p\gamma_p-2)(2-q\gamma_q)^{\frac{2-q\gamma_q}{p\gamma_p-2}}\left(\|\nabla u\|_2^2\right)^\frac{p\gamma_p-q\gamma_q}{p\gamma_p-2}}{\gamma_q(p\gamma_p-q\gamma_q)^{\frac{p\gamma_p-q\gamma_q}{p\gamma_p-2}}\gamma_p^{\frac{2-q\gamma_q}{p\gamma_p-2}}\|u\|_q^q\left(\|u\|_{p}^{p}\right)^{\frac{2-q\gamma_q}{p\gamma_p-2}}}.\endaligned\right.
\end{equation*}
\begin{proposition}\label{prop001}
Let $N\ge3$, $\mu>0$ and $2<q<2+\frac{4}{N}<p\leq2^*$.  Then for any $u\in \mathcal{S}_a$, $\Phi_{\mu,p,u}(s)$ is $C^\infty$ in $(0,\infty)$. Moreover, the set $\left\{\Phi_{\mu,p,u}(s)\mid u\in \mathcal{S}_a\right\}$
can be classified as follows:
 \begin{itemize}
    \item [$(a)$]\quad If $\mu\in\left(0,\mu_p(u)\right)$ then $\Phi_{\mu,p,u}(s)$ only has two critical points $0<t_{\mu,p}^{+}(u)<s_p(u)<t_{\mu,p}^{-}(u)<+\infty$, where $t_{\mu,p}^{+}(u)$ is the strict local minimum point of $\Phi_{\mu,p,u}(s)$ in $\left(0, t_{\mu,p}^{-}(u)\right)$ and $t_{\mu,p}^{-}(u)$ is the strict local maximum point of $\Phi_{\mu,p,u}(s)$ in $\left(t_{\mu,p}^{+}(u), +\infty\right)$.  Moreover, $(u)_{t_{\mu,p}^{\pm}(u)}\in\mathcal{P}_{a,\mu,p}^\pm$.
    \item [$(b)$]\quad If $\mu=\mu_p(u)$ then $\Phi_{\mu,p,u}(s)$ only has one degenerate critical point $s_p(u):=t_{\mu,p}^{0}(u)$.  Moreover, $(u)_{t_{\mu,p}^{0}(u)}\in\mathcal{P}_{a,\mu,p}^0$.
    \item [$(c)$]\quad If $\mu>\mu_p(u)$ then $\Phi_{\mu,p,u}(s)$ has no critical points.
 \end{itemize}
\end{proposition}
\begin{proof}
For any $u\in \mathcal{S}_a$, we define $h_{\mu,u}(s)=s^{2-q\gamma_q}\|\nabla u\|_2^2-\gamma_ps^{p\gamma_p-q\gamma_q}\|u\|_{p}^{p}$.  Direct calculations show that
\begin{eqnarray*}
\max_{s>0}h_{\mu,u}(s)&=&h_{\mu,u}\left(\left(\frac{(2-q\gamma_q)\|\nabla u\|_2^2}{(p\gamma_p-q\gamma_q)\gamma_p\|u\|_p^p}\right)^{\frac{1}{p\gamma_p-2}}\right)\\
&=&\left(\frac{(2-q\gamma_q)\|\nabla u\|_2^2}{(p\gamma_p-q\gamma_q)\gamma_p\|u\|_p^p}\right)^{\frac{2-q\gamma_q}{p\gamma_p-2}}\frac{(p\gamma_p-2)\|\nabla u\|_2^2}{p\gamma_p-q\gamma_q}.
\end{eqnarray*}
Thus, the equation
\begin{equation*}
s^2\|\nabla u\|_2^2-\mu\gamma_qs^{q\gamma_q}\|u\|_q^q-\gamma_ps^{p\gamma_p}\|u\|_{p}^{p}=0
\end{equation*}
has solutions if and only if $\mu\leq\mu_{p}(u)$, where it has two solutions for $\mu<\mu_{p}(u)$ satisfying $0<t_{\mu,p}^{+}(u)<s_p(u)<t_{\mu,p}^{-}(u)<+\infty$ and has only one solution $s_p(u)$ for  $\mu=\mu_{p}(u)$.
\end{proof}

\vskip0.12in

For the sake of simplicity, we denote 
\begin{eqnarray*}
\mu_a^*:=\mu_{a,2^*}^*\quad\text{and}\quad\mu(u):=\mu_{2^*}(u).
\end{eqnarray*}
Thus, 
\begin{equation}\label{defi001}
\mu^*_a=\tilde{C}_{N,q}\inf_{u\in \mathcal{S}_a}\frac{\left(\|\nabla u\|_2^2\right)^\frac{2^*-q\gamma_q}{2^*-2}}{\|u\|_q^q\left(\|u\|_{2^*}^{2^*}\right)^{\frac{2-q\gamma_q}{2^*-2}}}
\end{equation}
where
\begin{equation*}
 \tilde{C}_{N,q}=\frac{(2^*-2)(2-q\gamma_q)^{\frac{2-q\gamma_q}{2^*-2}}}{\gamma_q(2^*-q\gamma_q)^{\frac{2^*-q\gamma_q}{2^*-2}}}.
\end{equation*}
\begin{proposition}\label{propmup}
Let $N\ge3$, $u\in \mathcal{S}_a$ and $2<q<2+\frac{4}{N}<p\leq2^*$, then
\begin{itemize}
    \item [$(a)$]\quad $0<\mu^*_{a,p}<\infty$ and the functional $\mu_p(u)$ is $0$-homogeneous for the scaling $(u)_s$, that is, $\mu_p((u)_s)=\mu_p(u)$ for all $s>0$.
    \item [$(b)$]\quad $\left(\mu_{a,p}^*\right)^\frac{p\gamma_p-2}{p\gamma_p-q\gamma_q}=\left(\mu_{1,p}^*\right)^\frac{p\gamma_p-2}{p\gamma_p-q\gamma_q}a^{-(2(1-\gamma_p)+\frac{(p-q)N}{p}\frac{p\gamma_p-2}{p\gamma_p-q\gamma_q})\frac{p}{N(p-2)}}$ is strictly decreasing in terms of $a>0$.
    \item [$(c)$]\quad If $p<2^*$ then $\mu_{a,p}^*$ is achieved by some $u_0\in \mathcal{S}_a$, moreover, $\mathcal{P}^0_{a,\mu,p}=\emptyset$ for all $\mu\in(0,\mu^*_{a,p})$ and $\mathcal{P}^0_{a,\mu,p}\not=\emptyset$ for all $\mu\in[\mu^*_{a,p},\infty)$.
    \item [$(d)$]\quad $\lim_{p\uparrow2^*}\mu^*_{a.p}=\mu^*_a$.
 \end{itemize}
\end{proposition}
\begin{proof}
$(a)$\quad It follows immediately from direct calculations and the Gagliardo-Nirenberg inequality.  

\vskip0.06in

$(b)$\quad We first notice that by the conclusion~$(a)$,
\begin{align}\label{eqnWu0002}
\left(\mu_{a,p}^*\right)^\frac{p\gamma_p-2}{p\gamma_p-q\gamma_q}=\inf_{u\in \mathcal{S}_a,\|u\|_p^p=1}\left(\mu_p(u)\right)^\frac{p\gamma_p-2}{p\gamma_p-q\gamma_q}.
\end{align}
Moreover, for any $u_{a}\in \mathcal{S}_{a}$, we define $u_{1}=(u_{a})^s=s^\frac{N}{p}u_{a}(sx)$ with $s^{\frac{N(p-2)}{p}}=a$.  Then $u_{1}\in \mathcal{S}_{1}$, $\|u_{1}\|_p=\|u_{a}\|_p$ and 
\begin{equation*}
\left(\mu_p(u_1)\right)^\frac{p\gamma_p-2}{p\gamma_p-q\gamma_q}=a^{(2(1-\gamma_p)+\frac{(p-q)N}{p}\frac{p\gamma_p-2}{p\gamma_p-q\gamma_q})\frac{p}{N(p-2)}}\left(\mu_p(u_{a})\right)^\frac{p\gamma_p-2}{p\gamma_p-q\gamma_q}.
\end{equation*}
It follows that
\begin{eqnarray*}
\left(\mu_{1,p}^*\right)^\frac{p\gamma_p-2}{p\gamma_p-q\gamma_q}=a^{(2(1-\gamma_p)+\frac{(p-q)N}{p}\frac{p\gamma_p-2}{p\gamma_p-q\gamma_q})\frac{p}{N(p-2)}}\left(\mu_{a,p}^*\right)^\frac{p\gamma_p-2}{p\gamma_p-q\gamma_q}.
\end{eqnarray*}
Since $2<q<2+\frac{4}{N}<p\leq2^*$, we have $2(1-\gamma_p)+\frac{(p-q)N}{p}\frac{p\gamma_p-2}{p\gamma_p-q\gamma_q}>0$ which implies $\mu_{a,p}^*$ is strictly decreasing in terms of $a>0$.

\vskip0.06in

$(c)$\quad By the Schwartz symmetrization and \eqref{eqnWu0002}, we can choose a minimizing sequence, say $\{u_n\}$, which is radially symmetric and radially decreasing and $\|u_n\|_p=1$.  It follows that
\begin{equation*}
\left(\mu_{a,p}^*+o(1)\right)\|u_n\|_q^q=C\left(\|\nabla u_n\|_2^2\right)^\frac{p\gamma_p-q\gamma_q}{p\gamma_p-2},
\end{equation*}
which, together with the Gagliardo-Nirenberg inequality and $2<q<2+\frac{4}{N}$, implies that $\{u_n\}$ is bounded in $H^1(\bbr^N)$.  Thus, up to a subsequence, $u_n\rightharpoonup u_0$ weakly in $H^1(\bbr^N)$ as $n\to\infty$ for some $u_0\in H^1(\bbr^N)$ and by Strauss's radial lemma (cf. \cite[Lemma~A. I']{BerestyckiLions1983} or \cite[Lemma~1.24]{Willem1996}), we also have $u_n\to u_0$ strongly in $L^r(\bbr^N)$ as $n\to\infty$ for all $2<r<2^*$.  Clearly, by $\|u_n\|_p=1$, we know that $u_0\not=0$ and by the Fatou lemma, we have $0<\|u_0\|_2^2=a_1\le a$.  Let $v_n=u_n-u_0$ then by the conclusion~$(b)$, 
\begin{align}\label{002}
\left(\mu_{a,p}^*\right)^\frac{p\gamma_p-2}{p\gamma_p-q\gamma_q}=&\left(\mu_p(u_0)\right)^\frac{p\gamma_p-2}{p\gamma_p-q\gamma_q}+C\frac{\|\nabla v_n\|_2^2}{\left(\|u_0\|_q^q\right)^\frac{p\gamma_p-2}{p\gamma_p-q\gamma_q}}+o(1)\nonumber\\
\ge&\left(\mu_{a_1,p}^*\right)^\frac{p\gamma_p-2}{p\gamma_p-q\gamma_q}+C\frac{\|\nabla v_n\|_2^2}{\left(\|u_0\|_q^q\right)^\frac{p\gamma_p-2}{p\gamma_p-q\gamma_q}}+o(1)\nonumber\\
\geq&\left(\mu_{a,p}^*\right)^\frac{p\gamma_p-2}{p\gamma_p-q\gamma_q}+C\frac{\|\nabla v_n\|_2^2}{\left(\|u_0\|_q^q\right)^\frac{p\gamma_p-2}{p\gamma_p-q\gamma_q}}+o(1)
\end{align}
where $C>0$ is a constant.  It follows from \eqref{002} that $a_1=a$ and 
$v_n\to0$ strongly in $H^1(\bbr^N)$ as $n\to\infty$.  Thus, $\mu_{a,p}^*$ is achieved by $u_0\in \mathcal{S}_a$.  The remaining conclusions of $(c)$ follow immediately from Proposition~\ref{prop001}.

\vskip0.06in

$(d)$\quad For any $\epsilon>0$, by the definition of $\mu_{a}^*$, there exists $u_\epsilon\in \mathcal{S}_a$ such that $\mu(u_\epsilon)<\mu_{a}^*+\epsilon$.  Then by the dominated convergence theorem and the absolute continuity of $\|u_\epsilon\|_{2^*}$,
\begin{equation*}
\lim_{p\uparrow2^*}\mu_{a,p}^*\le\lim_{p\uparrow2^*}\mu_p(u_\epsilon)=\mu(u_\epsilon)<\mu_{a}^*+\epsilon,
\end{equation*}
which, together with the arbitrariness of $\epsilon>0$, implies that $\lim_{p\uparrow2^*}\mu_{a,p}^*\leq\mu_{a}^*$.  On the other hand, by the conclusions~$(a)$ and $(b)$, there exists $u_p\in \mathcal{S}_a$ such that 
\begin{eqnarray*}
\mu_p(u_p)=\mu_{a,p}^*\quad\text{and}\quad \|u_p\|_{2^*}^{2^*}=1.
\end{eqnarray*}
Since $|u_p|^p\leq |u_p|^2+|u_p|^{2^*}$, by $\mu_p(u_p)=\mu_{a,p}^*$ and $\lim_{p\uparrow 2^*}\mu_{a,p}^*\leq\mu_{a}^*$, we know that $\{u_p\}$ is bounded in $H^1(\bbr^N)$.  It follows from the H\"older inequality that
\begin{align*}
\mu_{a,p}^*=\tilde{C}_{N,q,p}\frac{\left(\|\nabla u_p\|_2^2\right)^\frac{p\gamma_p-q\gamma_q}{p\gamma_p-2}}{\|u_p\|_q^q\left(\|u_p\|_p^p\right)^\frac{2-q\gamma_q}{p\gamma_p-q\gamma_q}}\geq \tilde{C}_{N,q,p}\frac{\left(\|\nabla u_p\|_2^2\right)^\frac{2^*-q\gamma_q}{2^*-2}}{\|u_p\|_q^q\left(\|u_p\|_{2^*}^{2^*}\right)^{\frac{2-q\gamma_q}{2^*-2}}}+o(1)
\geq\mu_{a}^*+o(1),
\end{align*}
where
\begin{eqnarray*}
\tilde{C}_{N,q,p}=\frac{(p\gamma_p-2)(2-q\gamma_q)^{\frac{2-q\gamma_q}{p\gamma_p-2}}}{\gamma_q(p\gamma_p-q\gamma_q)^{\frac{p\gamma_p-q\gamma_q}{p\gamma_p-2}}\gamma_p^{\frac{2-q\gamma_q}{p\gamma_p-2}}}.
\end{eqnarray*}
Thus, we also have $\lim_{p\uparrow2^*}\mu_{a,p}^*\geq\mu_{a}^*$, which, together with $\lim_{p\uparrow2^*}\mu_{a,p}^*\leq\mu_{a}^*$, implies that $\lim_{p\uparrow2^*}\mu^*_{a,p}=\mu^*_a$.
\end{proof}

\vskip0.12in

\section{The existence of solutions of $(\ref{eqd106})$ for $\mu<\mu^*_a$}
\subsection{The existence of ground-state solutions}
In this section, we shall study the variational problem:
\begin{equation}\label{eqnWu0003}
m^+(a,\mu)=\inf_{u\in \mathcal{P}_{a, \mu}^+}\Psi_\mu(u).
\end{equation}
We begin with the following properties.
\begin{lemma}\label{lem001}
Let $N\geq3$, $2<q<2+\frac{4}{N}<p\leq 2^*$, $a>0$ and $\mu\in\left(0,\mu^*_{a,p}\right)$.  Then
\begin{itemize}
    \item [$(a)$]\quad $m^-(a,\mu)\ge m^+(a,\mu)$ and $m^+(a,\mu)<0$.
    \item [$(b)$]\quad for all $u_a\in\mathcal{S}_{a}$ and $b>a$ such that $\mu^*_{b,p}>\mu$, we have
\begin{eqnarray*}
\Psi_{\mu,p}(u_a)>\Psi_{\mu,p}\left(\left(\sqrt{\frac{b}{a}}u_{a}\right)_{t_{\mu,p}^\pm\left(\sqrt{\frac{b}{a}}u_{a}\right)}\right).
\end{eqnarray*}
In particular, $m^\pm(a,\mu)\ge m^\pm(b,\mu)$.
 \end{itemize}
\end{lemma}
\begin{proof}
$(a)$\quad Since $\mu\in\left(0,\mu^*_{a,p}\right)$, the conclusion follows immediately from Proposition~\ref{prop001}.

$(b)$\quad Since $\mu\in\left(0,\mu^*_{a,p}\right)$, by Proposition~\ref{prop001} and $(b)$ of Proposition~\ref{propmup}, there exist $t_{\mu,p}^\pm\left(\sqrt{\frac{b}{a}}u_{a}\right)>0$ such that $\left(\sqrt{\frac{b}{a}}u_{a}\right)_{t_{\mu,p}^\pm\left(\sqrt{\frac{b}{a}}u_{a}\right)}\in\mathcal{P}_{b,\mu,p}^\pm$.
The rest of the proof is similar to that of \cite[Lemma~3.2]{WeiWu2022} by computing the derivative of $\Psi_{\mu,p}\left(\left(\sqrt{\frac{b}{a}}u_{a}\right)_{t_{\mu,p}^\pm\left(\sqrt{\frac{b}{a}}u_{a}\right)}\right)$ in terms of $b$ with trivial modifications, so we omit it here.
\end{proof}

\vskip0.12in

With Lemma~\ref{lem001} in hands, we can prove the following.
\begin{proposition}\label{propWu0001}
Let $N\ge 3$, $2<q<2+\frac{4}{N}$, $a>0$ and $0<\mu<\mu^{*}_{a}$.  Then the variational problem~\eqref{eqnWu0003} is achieved by some $u_{a,\mu,+}$, which is real valued, positive, radially symmetric and radially decreasing.
Moreover, $u_{a,\mu,+}$ also satisfies the Schr\"{o}dinger equation~\eqref{eqd106} for a suitable Lagrange multiplier $\lambda=\lambda_{a, \mu, +}<0$.
\end{proposition}
\begin{proof}
Let $\{u_n\}\subset \mathcal{P}_{a,\mu}^+$ be a minimizing sequence of $m^+(a,\mu)$.  By the Schwartz symmetrization and Proposition~\ref{prop001}, we may assume that $\{u_n\}$ are real valued, nonnegative, radially symmetric and radially decreasing.  Since $\{u_n\}\subset \mathcal{P}_{a,\mu}^+$, we have
\begin{equation}\label{eqnWu0017}
(2^*-2)\|\nabla u_n\|^2_2-(2^*\gamma_q-q\gamma_q)\mu\gamma_q\|u_n\|_q^q<0.
\end{equation}
It follows from the Gagliardo-Nirenberg inequality that $\{u_n\}$ is bounded in $H^1(\bbr^N)$.  Thus, up to a subsequence, $u_n\rightharpoonup u_0$ weakly in $H^1(\bbr^N)$ as $n\to\infty$ for some $u_0\in H^1(\bbr^N)$ and by Strauss's radial lemma (cf. \cite[Lemma~A. I']{BerestyckiLions1983} or \cite[Lemma~1.24]{Willem1996}), we also have $u_n\to u_0$ strongly in $L^r(\bbr^N)$ as $n\to\infty$ for all $2<r<2^*$.  Clearly, $u_0$ is real valued, nonnegative, radially symmetric and radially decreasing.  Let $v_n=u_n-u_0$ then one of the following two cases must happen:
\begin{itemize}
    \item [$(i)$]\quad $u_0=0$.
    \item [$(ii)$]\quad $u_0\not=0$.
\end{itemize}
If the case~$(i)$ happens then $\|u_n\|_q^q\to 0$ as $n\to\infty$.  Moreover, by $\{u_n\}\subset \mathcal{P}_{a,\mu}^+$ and the Sobolev inequality, either $\|\nabla u_n\|_2^2=\|u_n\|^{2^*}_{2^*}=o(1)$ or $\|\nabla u_n\|_2^2=\|u_n\|^{2^*}_{2^*}+o(1)\ge S^{\frac{N}{2}}+o(1)$, which implies that either $\Psi_\mu(u_n)=o(1)$ or $\Psi_\mu(u_n)\ge\frac{1}{N}S^{\frac{N}{2}}+o(1)$.  These contradict $(a)$ of Lemma~\ref{lem001}.  Thus, the case~$(i)$ can not happen.  If the case~$(ii)$ happens then we claim that there exists $\delta>0$ sufficiently small such that
\begin{equation}\label{equ020}
\inf_{\mathcal{P}_{a,\mu}^{\delta,+}}\Psi_\mu(u)=\inf_{\mathcal{P}_{a,\mu}}\Psi_\mu(u)=\inf_{\mathcal{P}_{a,\mu}^+}\Psi_\mu(u)=m^+(a, \mu),
\end{equation}
where $\mathcal{P}_{a,\mu}^{\delta,+}=\{u\in \mathcal{S}_a\mid \text{dist}_{H^1}(u,\mathcal{P}_{a,\mu}^+)\le\delta\}$ with 
\begin{eqnarray*}
\text{dist}_{H^1}(u,\mathcal{P}_{a,\mu}^+)=\inf_{v\in\mathcal{P}_{a,\mu}^+}(\|\nabla u-\nabla v\|_2^2+\|u-v\|_2^2)^{\frac12}.
\end{eqnarray*}
Suppose the contrary.  Then by $(a)$ of Lemma~\ref{lem001}, there exists $\delta_n\to0$ as $n\to\infty$,  $w_n\in \mathcal{S}_a$ and $\phi_n\in \mathcal{P}_{a,\mu}^+$ such that $w_n-\phi_n\to0$ strongly in $H^1(\bbr^N)$ as $n\to\infty$ and $\Psi_\mu(w_n)<m^+(a,\mu)$ for all $n$.
Since $0<\mu<\mu^{*}_{a}$, by Proposition~\ref{prop001}, there exists $t_{\mu,2^*}^+(w_n)>0$ such that $\left(w_n\right)_{t^+_{\mu,2^*}(w_n)}\in\mathcal{P}_{a,\mu}^+$.  It follows from $\phi_n\in \mathcal{P}_{a,\mu}^+$ and $w_n-\phi_n\to0$ strongly in $H^1(\bbr^N)$ as $n\to\infty$ that
\begin{equation}\label{026}
\left\{\array{ll}\left(t_{\mu,2^*}^+\left(w_n\right)\right)^{2-q\gamma_q}\|\nabla w_n\|_2^2=\mu\gamma_q\|w_n\|_q^q+\left(t_{\mu,2^*}^+\left(w_n\right)\right)^{2^*-q\gamma_q}\|w_n\|_{2^*}^{2^*},
\\
\|\nabla w_n\|_2^2=\mu\gamma_q\|w_n\|_q^q+\|w_n\|_{2^*}^{2^*}+o(1).\endarray\right.
\end{equation}
By \eqref{026}, it is easy to see that $t_{\mu,2^*}^+\left(w_n\right)\sim1$.  Thus, up to a subsequence, we may assume that $t_{\mu,2^*}^+\left(w_n\right)\to t_{\mu,2^*}$ as $n\to\infty$.  It follows from \eqref{026} and $w_n-\phi_n\to0$ strongly in $H^1(\bbr^N)$ as $n\to\infty$ that
\begin{equation*}
\left(\left(t_{\mu,2^*}^+\left(w_n\right)\right)^{2-q\gamma_q}-1\right)\|\nabla \phi_n\|_2^2-\left(\left(t_{\mu,2^*}^+\left(w_n\right)\right)^{2^*-q\gamma_q}-1\right)\|\phi_n\|_{2^*}^{2^*}=o(1).
\end{equation*}
As in the case~$(i)$, we know that $\|\phi_n\|_{2^*}\gtrsim1$, which implies that
\begin{equation}\label{022}
\frac{\|\nabla \phi_n\|_2^2}{\|\phi_n\|_{2^*}^{2^*}}=\frac{\left(t_{\mu,2^*}^+\left(w_n\right)\right)^{2^*-q\gamma_q}-1}{\left(t_{\mu,2^*}^+\left(w_n\right)\right)^{2-q\gamma_q}-1}+o(1).
\end{equation}
By Proposition~\ref{prop001} and $0<\mu<\mu^{*}_{a}$, 
\begin{equation*}
1=t_{\mu,2^*}^+\left(\phi_n\right)<\liminf_{n\to\infty}s_{2^*}\left(\phi_n\right)=\liminf_{n\to\infty}\left(\frac{(2-q\gamma_q)\|\nabla\phi_n\|_2^2}{(2^*-q\gamma_q)\|\phi_n\|_{2^*}^{2^*}}\right)^{\frac{1}{2^*-2}},
\end{equation*}
which, together with Proposition~\ref{prop001} and \eqref{022}, implies that
\begin{equation*}
\frac{2^*-q\gamma_q}{2-q\gamma_q}<\liminf_{n\to\infty}\frac{\left(t_{\mu,2^*}^+\left(w_n\right)\right)^{2^*-q\gamma_q}-1}{\left(t_{\mu,2^*}^+\left(w_n\right)\right)^{2-q\gamma_q}-1}.
\end{equation*}
Since $g(t)=\frac{t^{2^*-q\gamma_q}-1}{t^{2-q\gamma_q}-1}$ is increasing in $(0, +\infty)$ and $\lim_{t\to1}g(t)=\frac{2^*-q\gamma_q}{2-q\gamma_q}$, we must have $t_{\mu,2^*}>1$ which, together with Proposition~\ref{prop001}, implies that
$\Psi_\mu(w_n)\ge\Psi_\mu\left((w_n)_{t^+_{\mu,2^*}(w_n)}\right)$ for $n$ sufficient large.  It follows that
\begin{equation*}
m^+(a,\mu)\le\Psi_\mu\left((w_n)_{t^+_{\mu,2^*}(w_n)}\right)\le\Psi_\mu(w_n)<m^+(a,\mu),
\end{equation*}
which is impossible for $n$ sufficient large.  Thus, \eqref{equ020} holds true for $\delta>0$ sufficiently small.  Now, since $\mathcal{P}_{a,\mu}^{\delta,+}$ is closed in the $H^1(\bbr^N)$ topology, $\Psi_\mu(u)$ is bounded from below by \eqref{equ020} and $\{u_n\}$ is a minimizing sequence, by Ekeland's variational principle, there exists $\{\tilde{w}_n\}\subset\mathcal{P}_{a,\mu}^{\delta,+}$ such that $\{\tilde{w}_n\}$ is a $(PS)_{m^+(a, \mu)}$ sequence of $\Psi_\mu(u)|_{\mathcal{S}_a}$ and $\tilde{w}_n-u_n\to0$ strongly in $H^1(\bbr^N)$ as $n\to\infty$.  Thus, up to a subsequence, $\tilde{w}_n\rightharpoonup u_0$ weakly in $H^1(\bbr^N)$ as $n\to\infty$.  It follows from the method of the Lagrange multiplier that there exists $\lambda_n\in\bbr$ such that
\begin{equation*}
-\Delta\tilde{w}_n-\lambda_n\tilde{w}_n-\mu|\tilde{w}_n|^{q-2}\tilde{w}_n-|\tilde{w}_n|^{2^*-2}\tilde{w}_n=o(1)
\end{equation*}
strongly in $H^{-1}(\bbr^N)$ as $n\to\infty$ where $H^{-1}(\bbr^N)$ is the dual space of $H^{1}(\bbr^N)$.  Thus, by $\{\tilde{w}_n\}\subset\mathcal{P}_{a,\mu}^{\delta,+}$, $\{u_n\}\subset \mathcal{P}_{a,\mu}^+$ and $\tilde{w}_n-u_n\to0$ strongly in $H^1(\bbr^N)$ as $n\to\infty$, we have
\begin{align*}
\lambda_n a=&\|\nabla \tilde{w}_n\|_2^2-\mu\|\tilde{w}_n\|_q^q-\|\tilde{w}_n\|_{2^*}^{2^*}\\
=&(\gamma_q-1)\mu\|u_n\|_q^q+o(1)\\
=&(\gamma_q-1)\mu\|u_0\|_q^q+o(1),
\end{align*}
which, together with $u_0\not=0$ and $\gamma_q<1$ for $2<q<2+\frac{4}{N}$, implies that $\lambda_n\to\lambda_0<0$ as $n\to\infty$ up to a subsequence.  It follows that $u_0$ is a weak solution of the following equation
\begin{equation*}
-\Delta u=\lambda_0u+\mu|u|^{q-2}u+|u|^{2^*-2}u\quad\text{in }\bbr^N.
\end{equation*}
By the standard elliptic regularity theorem, we know that $u_0$ is smooth and exponentially decays to zero at infinity.  Thus, $u_0\in\mathcal{P}_{a_1,\mu}$.  Recall that $u_n\to u_0$ strongly in $L^r(\bbr^N)$ as $n\to\infty$ for all $2<r<2^*$ and $\{u_n\}\subset \mathcal{P}_{a,\mu}^+$, thus, we must have
\begin{equation*}
\|\nabla v_n\|_2^2-\|v_n\|^{2^*}_{2^*}=o(1).
\end{equation*}
Similar to the case~$(i)$, either $\|\nabla v_n\|_2^2=\|v_n\|^{2^*}_{2^*}=o(1)$ or $\|\nabla v_n\|_2^2=\|v_n\|^{2^*}_{2^*}+o(1)\ge S^{\frac{N}{2}}+o(1)$.  Moreover, by the Fatou lemma, we also have $0<\|u_0\|_2^2=a_1\le a$.  It follows from the Brezis-Lieb lemma (cf. \cite[Lemma~1.32]{Willem1996}) and $(b)$ of Lemma~\ref{lem001} that 
\begin{align}\label{eqnWu0005}
m^+(a, \mu)+o(1)=&\frac{1}{2}\|\nabla u_n\|_2^2-\frac{1}{2^*}\|u_n\|_{2^*}^{2^*}-\frac{\mu}{q}\|u_n\|_q^q\notag\\
=&\Psi_\mu(u_0)+\frac{1}{2}\|\nabla v_n\|_2^2-\frac{1}{2^*}\|v_n\|^{2^*}_{2^*}+o(1)\notag\\
\ge&m^+(a_1, \mu)+o(1)\notag\\
\ge&m^+(a, \mu)+o(1),
\end{align}
which implies that $u_0$ is a minimizer of $m^+(a_1, \mu)$.  If $a_1<a$ then by $(b)$ of Lemma~\ref{lem001} once more, we have $m^+(a_1, \mu)>m^+(a, \mu)$ which contradicts \eqref{eqnWu0005}.  Thus, we must have $a_1=a$ and $v_n\to0$ strongly in $H^1(\bbr^N)$ as $n\to\infty$.  By denoting $u_{a,\mu,+}=u_0$ and $\lambda_{a,\mu,+}=\lambda_0$, we know that the variational problem~\eqref{eqnWu0003} is achieved by $u_{a,\mu,+}$ which, by the method of the Lagrange multiplier, also satisfies the Schr\"{o}dinger equation~\eqref{eqd106} for a suitable $\lambda_{a, \mu, +}<0$.  Since $u_{a,\mu,+}$ is real valued, nonnegative, radially symmetric and radially decreasing, by the maximum principle, we also know that $u_{a,\mu,+}$ is also positive.
\end{proof}

\vskip0.12in

\subsection{The existence of mountain-pass solutions}
In this section, we shall study the variational problem:
\begin{equation}\label{eqnWu0004}
m^-(a,\mu)=\inf_{u\in \mathcal{P}_{a, \mu}^-}\Psi_\mu(u).
\end{equation}
We begin with the proof of Proposition~\ref{003}.

\vskip0.12in

\noindent\textbf{Proof of Proposition~\ref{003}:}\quad 
Let $\{u_n\}\in \mathcal{P}_{a,\mu,p}^-$ be a minimizing sequence of $m^-_p(a,\mu)$.  As in the proof of Proposition~\ref{propWu0001}, we may assume that $\{u_n\}$ are real valued, nonnegative, radially symmetric and radially decreasing.  Since $2<q<2+\frac{4}{N}<p<2^*$, by similar estimates in the proof of \cite[Lemma~4.1]{Soave2020-1}, we know that $\{u_n\}$ is bounded in $H^1(\bbr^N)$.  Again, as in the proof of Proposition~\ref{propWu0001}, up to a subsequence, $u_n\rightharpoonup u_0$ weakly in $H^1(\bbr^N)$ as $n\to\infty$ for some $u_0\in H^1(\bbr^N)$ and by Strauss's radial lemma (cf. \cite[Lemma~A. I']{BerestyckiLions1983} or \cite[Lemma~1.24]{Willem1996}), we also have $u_n\to u_0$ strongly in $L^r(\bbr^N)$ as $n\to\infty$ for all $2<r<2^*$.  Clearly, $u_0$ is real valued, nonnegative, radially symmetric and radially decreasing.  Let $v_n=u_n-u_0$ then one of the following two cases must happen:
\begin{itemize}
    \item [$(i)$]\quad $u_0=0$.
    \item [$(ii)$]\quad $u_0\not=0$.
\end{itemize}
If the case~$(i)$ happens then by $\{u_n\}\in \mathcal{P}_{a,\mu,p}^-$, we have $\|\nabla u_n\|_2^2=o(1)$.  However, since $2<q<2+\frac{4}{N}<p<2^*$, by $\{u_n\}\in \mathcal{P}_{a,\mu,p}^-$ once more and the Gagliardo-Nirenberg inequality, we also have
\begin{equation}\label{0003}
\|\nabla u_n\|_2^2\lesssim\|u_n\|_p^p\lesssim\|\nabla u_n\|_2^{p\gamma_p},
\end{equation}
which implies that $\|\nabla u_n\|_2^2\gtrsim1$.  Thus, the case~$(i)$ can not happen.
If the case~$(ii)$ happens then by the Fatou lemma, we have $0<\|u_0\|_2^2=a_1\le a$.
Moreover, since $0<\mu<\mu^*_{a,p}$, by Propsoition~\ref{prop001}, there exists $t^+_{\mu,p}(u_n)$ and $t_{\mu,p}^-(u_0)$ such that $(u_n)_{t^+_{\mu,p}(u_n)}\in \mathcal{P}_{a,\mu,p}^+$ and $(u_0)_{t_{\mu,p}^-(u_0)}\in\mathcal{P}_{a_1,\mu,p}^-$.  Since
\begin{align*}
0=&\left(t_{\mu,p}^-(u_0)\right)^2\|\nabla u_0\|_2^2-\left(t_{\mu,p}^-(u_0)\right)^{2^*}\|u_0\|_{2^*}^{2^*}-\mu\gamma_q\left(t_{\mu,p}^-(u_0)\right)^{q\gamma_q}\|u_0\|_q^q\nonumber\\
\leq&\left(t_{\mu,p}^-(u_0)\right)^2\|\nabla u_n\|_2^2-\left(t_{\mu,p}^-(u_0)\right)^{2^*}\|u_n\|_{2^*}^{2^*}-\mu\gamma_q\left(t_{\mu,p}^-(u_0)\right)^{q\gamma_q}\|u_n\|_q^q+o(1),
\end{align*}
by Propsoition~\ref{prop001}, we know that $t^+_{\mu,p}(u_n)+o(1)\le t_{\mu,p}^-(u_0)\le1+o(1)$.  Moreover, by similar estimates of \eqref{0003} and the fact that $u_0\not=0$, we also have $t_{\mu,p}^-(u_0)\gtrsim1$.  It follows from Propsoition~\ref{prop001} once more and $(b)$ of Lemma~\ref{lem001} that
\begin{align}\label{eqnWu0006}
m_p^-(a,\mu)+o(1)=&\Psi_{\mu,p}(u_n)\notag\\
\geq&\Psi_{\mu,p}\left((u_n)_{t_{\mu,p}^-(u_0)}\right)+o(1)\notag\\
=&\Psi_{\mu,p}\left((u_0)_{t_{\mu,p}^-(u_0)}\right)+\frac{1}{2}\left(t^-_{\mu,p}(u_0)\right)^2\|\nabla v_n\|_2^2+o(1)\notag\\
\ge& m_p^-(a_1,\mu)+\frac{1}{2}\left(t^-_{\mu,p}(u_0)\right)^2\|\nabla v_n\|_2^2+o(1)\notag\\
\geq&m_p^-(a,\mu)+\frac{1}{2}\left(t^-_{\mu,p}(u_0)\right)^2\|\nabla v_n\|_2^2+o(1).
\end{align}
As that in the proof of Proposition~\ref{propWu0001}, by \eqref{eqnWu0006}, we must have that $a_1=a$ and $v_n\to0$ strongly in $H^1(\bbr^N)$ as $n\to\infty$.  By denoting $u_{a,\mu,p,-}=u_0$, we know that the variational problem~\eqref{eqnWu0003} is achieved by $u_{a,\mu,p,-}$ which, by the method of the Lagrange multiplier, also satisfies the Schr\"{o}dinger equation~\eqref{eqd106} for a suitable $\lambda_{a, \mu, p,-}<0$.  Since $u_{a,\mu, p,-}$ is real valued, nonnegative, radially symmetric and radially decreasing, by the maximum principle, we also know that $u_{a,\mu, p,-}$ is also positive.
\hfill$\Box$

\vskip0.12in

\begin{remark}\label{rmkWu0002}
The variational problem~\eqref{eqnWu0009} has also been studied in \cite{Soave2020-1} under the restriction
\begin{equation*}
\frac{\left(\mu a^{q(1-\gamma_q)}\right)^{p\gamma_p-2}\left(a^{p(1-\gamma_p)}\right)^{2-q\gamma_q}}{\left(\frac{p(2-q\gamma_q)}{2C_{N,p}^p(p\gamma_p-q\gamma_q)}\right)^{2-q\gamma_q}\left(\frac{q(p\gamma_p-2)}{2C_{N,q}^q(p\gamma_p-q\gamma_q)}\right)^{p\gamma_p-2}}<1.
\end{equation*}
However, by the Gagliardo-Nirenberg inequality and the definition of the extremal value $\mu^*_{a,p}$ given by \eqref{eqnWu0025}, we have
\begin{equation*}
\frac{\left(\mu^*_{a,p}a^{q(1-\gamma_q)}\right)^{p\gamma_p-2}\left(a^{p(1-\gamma_p)}\right)^{2-q\gamma_q}}{\left(\frac{(2-q\gamma_q)}{\gamma_pC_{N,p}^p(p\gamma_p-q\gamma_q)}\right)^{2-q\gamma_q}\left(\frac{(p\gamma_p-2)}{\gamma_qC_{N,q}^q(p\gamma_p-q\gamma_q)}\right)^{p\gamma_p-2}}\geq1.
\end{equation*}
Direct calculations show that
\begin{equation*}
\frac{\left(\frac{(2-q\gamma_q)}{\gamma_pC_{N,p}^p(p\gamma_p-q\gamma_q)}\right)^{2-q\gamma_q}\left(\frac{(p\gamma_p-2)}{\gamma_qC_{N,q}^q(p\gamma_p-q\gamma_q)}\right)^{p\gamma_p-2}}{\left(\frac{p(2-q\gamma_q)}{2C_{N,p}^p(p\gamma_p-q\gamma_q)}\right)^{2-q\gamma_q}\left(\frac{q(p\gamma_p-2)}{2C_{N,q}^q(p\gamma_p-q\gamma_q)}\right)^{p\gamma_p-2}}=\left(\frac{2}{p\gamma_p}\right)^{2-q\gamma_q}\left(\frac{2}{q\gamma_q}\right)^{p\gamma_p-2},
\end{equation*}
where 
\begin{eqnarray*}
\min_{(p,q)\in\left[2+\frac{4}{N}, 2^*\right]\times\left[2, 2+\frac{4}{N}\right]}\left(\frac{2}{p\gamma_p}\right)^{2-q\gamma_q}\left(\frac{2}{q\gamma_q}\right)^{p\gamma_p-2}=\left(\frac{2}{2^*}\right)^{2-2}\left(\frac{2}{2}\right)^{2^*-2}=1.
\end{eqnarray*}
Thus, Proposition~\ref{003} improves \cite[Theorem~1.3]{Soave2020-1}.
\end{remark}

\vskip0.12in

We also need the following property in the construction of good minimizing sequence of the variational problem~\eqref{eqnWu0004}.
\begin{lemma}\label{lem002}
Let $N\geq3$, $2<q<2+\frac{4}{N}$, $a>0$ and $\mu\in\left(0,\mu^*_{a}\right)$.  Then
\begin{itemize}
    \item [$(a)$]\quad $m^-(a,\mu)<m^+(a,\mu)+\frac{1}{N}S^{\frac{N}{2}}$.
    \item [$(b)$]\quad For any sequence $\{p_n\}$ satisfying $p_n\uparrow2^*$ as $n\to\infty$, we have 
    \begin{eqnarray}\label{eqnWu0007}
    \limsup_{n\to +\infty}m^-_{p_n}(a,\mu)\le m^-(a,\mu).
    \end{eqnarray}
\end{itemize}
\end{lemma}
\begin{proof}
$(a)$\quad  As in the proof of \cite[Lemma~3.1]{WeiWu2022}, we define $\widehat{W}_{\ve,t}=u_{a,\mu,+}+tW_\ve$ and $\overline{W}_{\ve,t}=s^{\frac{N-2}{2}}\widehat{W}_{\ve,t}(sx)$ where 
$W_\ve =\chi (x) U_\epsilon$, with $ \chi (x)$ being a suitable cut-off function around $0$ and $U_\epsilon$ being the standard Aubin-Talenti bubble given by
\begin{eqnarray*}
U_\ve(x)=[N(N-2)]^{\frac{N-2}{4}}\bigg(\frac{\ve}{\ve^2+|x|^2}\bigg)^{\frac{N-2}{2}},
\end{eqnarray*}
and $s=\frac{\|\widehat{W}_{\ve,t}\|_2}{\sqrt{a}}$.  Then $\overline{W}_{\ve,t}\in\mathcal{S}_a$ for all $\ve,t>0$.  By Proposition~\ref{prop001}, there exist $\tau_{\ve,t}(\overline{W}_{\ve,t})>0$ such that $(\overline{W}_{\ve,t})_{\tau_{\ve,t}(\overline{W}_{\ve,t})}\in\mathcal{P}^-_{a,\mu}$ and as in the proof of \cite[Lemma~3.1]{WeiWu2022} again, we can prove that there exists $t_\ve>0$ such that $\tau_{\ve, t_\ve}(\overline{W}_{\ve,t_\ve})=1$.  It follows that 
\begin{eqnarray*}
m^-\left(a,\mu\right)\leq\sup_{t\geq0}\Psi_\mu(\overline{W}_{\ve,t}).
\end{eqnarray*}
The rest of the proof is to estimate $\sup_{t\geq0}\Psi_\mu(\overline{W}_{\ve,t})$ which is the same as that of \cite[Lemma~3.1 and Remark~3.1]{WeiWu2022} with trivial modifications, so we omit it here.

\vskip0.06in

$(b)$\quad Let $\{p_n\}$ be a sequence satisfying $p_n\uparrow2^*$ as $n\to\infty$.  For any fixed $u\in\mathcal{P}_{a,\mu}^-$, by Proposition~\ref{prop001} and $(d)$ of Proposition~\ref{propmup}, for $n$ sufficiently large, there exists $t_{\mu,p_n}^{-}(u)$ such that $\left(u\right)_{t_{\mu,p_n}^{-}(u)}\in\mathcal{P}_{a,\mu,p_n}^-$.  Since $2<q<2+\frac{4}{N}$ and $p_n\uparrow2^*>2+\frac{4}{N}$ as $n\to\infty$, by the dominated convergence theorem, it is easy to see that
\begin{equation*}
\left(\frac{\mu\gamma_q\|u\|_q^q}{\|\nabla u\|_2^2}\right)^\frac{1}{2-q\gamma_q}<t_{\mu,p_n}^{-}(u)<\left(\frac{\|\nabla u\|_2^2}{\gamma_{p_n}\|u\|_{p_n}^{p_n}}\right)^\frac{1}{p_n\gamma_{p_n}-2}=\left(\frac{\|\nabla u\|_2^2}{\|u\|_{2^*}^{2^*}}\right)^\frac{1}{2^*-2}+o(1).
\end{equation*}
Thus, up to a subsequence, we may assume that $t_{\mu,p_n}^{-}(u)\to t_{\mu,2^*}(u)$ which satisfies 
\begin{equation*}
\left\{\array{ll}
\left(t_{\mu,2^*}(u)\right)^2\|\nabla u\|_2^2-\left(t_{\mu,2^*}(u)\right)^{2^*}\|u\|_{2^*}^{2^*}-\mu\gamma_q\left(t_{\mu,2^*}(u)\right)^{q\gamma_q}\|u\|_q^q=0,
\\
2\left(t_{\mu,2^*}(u)\right)^2\|\nabla u\|_2^2-2^*\left(t_{\mu,2^*}(u)\right)^{2^*}\|u\|_{2^*}^{2^*}+\mu q\gamma_q^2\left(t_{\mu,2^*}(u)\right)^{q\gamma_q}\|u\|_q^q\le0.\endarray\right.
\end{equation*}
Since $\mu\in\left(0,\mu^*_{a}\right)$, by Proposition~\ref{prop001}, we must have $t_{\mu,2^*}(u)=1$.  It follows that
\begin{equation*}
m^-(a,p_n)\le \Psi_{\mu,p_n}\left(\left(u\right)_{t_{\mu,p_n}^{-}(u)}\right)=\Psi_{\mu}\left(u\right)+o(1).
\end{equation*}
Since $u\in\mathcal{P}_{a,\mu}^-$ is arbitrary, \eqref{eqnWu0007} holds true.
\end{proof}

\vskip0.12in

With Proposition~\ref{003} and Lemma~\ref{lem002} in hands, we can prove the following.
\begin{proposition}\label{propWu0002}
Let $N\ge 3$, $2<q<2+\frac{4}{N}$, $a>0$ and $0<\mu<\mu^*_{a}$.  Then the variational problem~\eqref{eqnWu0004} is achieved by some $u_{a,\mu,-}$, which is real valued, positive, radially symmetric and radially decreasing.
Moreover, $u_{a,\mu,-}$ also satisfies the Schr\"{o}dinger equation~\eqref{eqd106} for a suitable Lagrange multiplier $\lambda=\lambda_{a, \mu,-}<0$.
\end{proposition}
\begin{proof}
Let $\{p_n\}$ be a sequence satisfying $p_n\uparrow2^*$ as $n\to\infty$.  Then by Proposition~\ref{003} and $(d)$ of Proposition~\ref{propmup}, for $n$ sufficiently large, the variational problem~\eqref{eqnWu0009} is achieved by some $u_{a,\mu,p_n,-}$ which is also a solution of \eqref{eqnWu0008} for a suitable Lagrange multiplier $\lambda_{a, \mu,p_n,-}<0$.  By $(b)$ of Lemma~\ref{lem002} and similar estimates in the proof of \cite[Lemma~3.2]{Soave2020-2}, we know that $\{u_{a,\mu,p_n,-}\}$ is bounded in $H^1(\bbr^N)$.  Thus, up to a subsequence, $u_{a,\mu,p_n,-}\rightharpoonup u_{a,\mu,-}$ weakly in $H^1(\bbr^N)$ as $n\to\infty$ for some $u_{a,\mu,-}\in H^1(\bbr^N)$ and by Strauss's radial lemma (cf. \cite[Lemma~A. I']{BerestyckiLions1983} or \cite[Lemma~1.24]{Willem1996}), we also have $u_{a,\mu,p_n,-}\to u_{a,\mu,-}$ strongly in $L^r(\bbr^N)$ as $n\to\infty$ for all $2<r<2^*$.  Clearly, $u_{a,\mu,-}$ is real valued, nonnegative, radially symmetric and radially decreasing.  Since $u_{a,\mu,p_n,-}$ is a solution of \eqref{eqnWu0008} for a suitable Lagrange multiplier $\lambda_{a, \mu,p_n,-}<0$ and $u_{a,\mu,p_n,-}\in \mathcal{P}_{a,\mu,p_n}^-$, we must have
\begin{equation}\label{eqnWu0010}
\lambda_{a, \mu,p_n,-}a=\mu(\gamma_q-1)\|u_{a,\mu,p_n,-}\|_q^q+(\gamma_{p_n}-1)\|u_{a,\mu,p_n,-}\|_{p_n}^{p_n},
\end{equation}
which, together with the boundedness of $\{u_{a,\mu,p_n,-}\}$ in $H^1(\bbr^N)$ and the H\"older inequality, implies that $\{\lambda_{a, \mu,p_n,-}\}$ is bounded.  Thus, we may assume that 
\begin{eqnarray*}
\lim_{n\to\infty}\lambda_{a, \mu,p_n,-}=\lambda_{a, \mu,-}\leq0.
\end{eqnarray*}
Since $\gamma_{p_n}\to1$ as $p_n\to2^*$, by \eqref{eqnWu0010}, we also have
\begin{equation*}
\lambda_{a, \mu,-} a=\mu(\gamma_q-1)\|u_{a,\mu,-}\|_q^q.
\end{equation*}
It follows that $\lambda_{a, \mu,-}=0$ if and only if $u_{a,\mu,-}=0$.  Moreover, by the dominated convergence theorem, we also know that $u_{a,\mu,-}$ satisfies
\begin{equation*}
-\Delta u_{a,\mu,-}=\lambda_{a, \mu,-}u_{a,\mu,-}+\mu|u_{a,\mu,-}|^{q-2}u_{a,\mu,-}+|u_{a,\mu,-}|^{2^*-2}u_{a,\mu,-}\quad\text{in }\bbr^N
\end{equation*}
in the weak sense.  Since $u_{a,\mu,-}\in H^1(\bbr^N)$, by the standard elliptic regularity theorem, we have $u_{a,\mu,-}\in L^{\infty}_{loc}(\bbr^N)$.  Thus, by \cite[Proposition~1]{BerestyckiLions1983}, $u_{a,\mu,-}$ satisfies the Pohozaev identity, that is,
\begin{eqnarray*}
\|\nabla u_{a,\mu,-}\|_2^2-\|u_{a,\mu,-}\|_{2^*}^{2^*}-\mu\gamma_q\|u_{a,\mu,-}\|_q^q=0,
\end{eqnarray*}
which, together with $u_{a,\mu,p_n,-}\in \mathcal{P}_{a,\mu,p_n}^-$ and the H\"older inequality, implies that
\begin{equation}\label{010}
\|\nabla v_n\|_2^2\le\|v_n\|_{2^*}^{2^*}+o(1),
\end{equation}
where $v_n=u_{a,\mu,p_n,-}-u_{a,\mu,-}$.  As in the proof of Proposition~\ref{propWu0001}, by \eqref{010}, either $\|v_n\|^{2^*}_{2^*}=\|\nabla v_n\|_2^2=o(1)$ or $\|v_n\|^{2^*}_{2^*}\geq\|\nabla v_n\|_2^2+o(1)\ge S^{\frac{N}{2}}+o(1)$.
Note that by the Fatou lemma, we also have $\|u_{a,\mu,-}\|_2^2=a_1\le a$.  As in the proof of Proposition~\ref{propWu0001}, one of the following two cases must happen:
\begin{itemize}
    \item [$(i)$]\quad $u_{a,\mu,-}=0$.
    \item [$(ii)$]\quad $u_{a,\mu,-}\not=0$.
\end{itemize}
If the case~$(i)$ happens then either 
\begin{eqnarray}\label{eqnWu0011}
\|u_{a,\mu,p_n,-}\|^{2^*}_{2^*}=\|\nabla u_{a,\mu,p_n,-}\|_2^2=o(1)
\end{eqnarray}
or 
\begin{eqnarray}\label{eqnWu0012}
\|u_{a,\mu,p_n,-}\|^{2^*}_{2^*}\geq\|\nabla u_{a,\mu,p_n,-}\|_2^2+o(1)\ge S^{\frac{N}{2}}+o(1).
\end{eqnarray}
Since $u_{a,\mu,p_n,-}\in \mathcal{P}_{a,\mu,p_n}^-$, we have
\begin{equation*}
\left\{\array{ll}\|\nabla u_{a,\mu,p_n,-}\|_2^2-\gamma_{p_n}\|u_{a,\mu,p_n,-}\|_{p_n}^{p_n}-\mu\gamma_q\|u_{a,\mu,p_n,-}\|_q^q=0,
\\
2\|\nabla u_{a,\mu,p_n,-}\|_2^2-p_n\gamma_{p_n}^2\|u_{a,\mu,p_n,-}\|_{p_n}^{p_n}+\mu q\gamma_q^2\|u_{a,\mu,p_n,-}\|_q^q<0,\endarray\right.
\end{equation*}
which, together with the Gagliardo-Nirenberg inequality and \eqref{eqnWu0011}--\eqref{eqnWu0012}, implies that
\begin{eqnarray*}
\|u_{a,\mu,p_n,-}\|^{2^*}_{2^*}\geq\|\nabla u_{a,\mu,p_n,-}\|_2^2+o(1)\ge S^{\frac{N}{2}}+o(1).
\end{eqnarray*}
It follows from Lemma~\ref{lem002} that
\begin{align*}
m^-(a,\mu)\ge&\limsup_{n\to +\infty}m_{p_n}^-(a,\mu)\\
=&\limsup_{n\to+\infty}\Psi_{\mu,p_n}(u_{a,\mu,p_n,-})\\
=&\limsup_{n\to+\infty}\left(\left(\frac{1}{2}-\frac{1}{2^*}\right)\|\nabla u_{a,\mu,p_n,-}\|_2^2-\mu\gamma_q\left(\frac{1}{q\gamma_q}-\frac{1}{2^*}\right)\|u_{a,\mu,p_n,-}\|_q^q\right)\\
\ge&\frac{1}{N}S^\frac{N}{2},
\end{align*}
which contradicts Lemmas~\ref{lem001} and \ref{lem002}.  Thus, the case~$(i)$ can not happen.  If the case~$(ii)$ happens then $0<a_1\leq a$ and $\lambda_{a, \mu,-}<0$.  Now, by the Brezis-Lieb lemma (cf. \cite[Lemma~1.32]{Willem1996}) and Lemmas~\ref{lem001} and \ref{lem002},
\begin{align*}
m^-(a,\mu)\ge&\limsup_{n\to +\infty}m_{p_n}^-(a,\mu)\\
=&\limsup_{n\to+\infty}\Psi_{\mu,p_n}(u_{a,\mu,p_n,-})\\
=&\limsup_{n\to+\infty}\left(\left(\frac{1}{2}-\frac{1}{2^*}\right)\|\nabla u_{a,\mu,p_n,-}\|_2^2-\mu\gamma_q\left(\frac{1}{q\gamma_q}-\frac{1}{2^*}\right)\|u_{a,\mu,p_n,-}\|_q^q\right)\\
=&\Psi_{\mu}(u_{a,\mu,-})+\left(\frac{1}{2}-\frac{1}{2^*}\right)\limsup_{n\to\infty}\|\nabla v_n\|_2^2+o(1)\\
\ge&m^+(a_1,\mu)+\frac{1}{N}S^\frac{N}{2}\chi_{\|\nabla v_n\|_2^2\ge S^{\frac{N}{2}}+o(1)}\\
\ge&m^+(a,\mu)+\frac{1}{N}S^\frac{N}{2}\chi_{\|\nabla v_n\|_2^2\ge S^{\frac{N}{2}}+o(1)},
\end{align*} 
which contradicts Lemma~\ref{lem001} unless $a_1=a$ and $v_n\to0$ strongly in $H^1(\bbr^N)$ as $n\to\infty$.  Here, 
\begin{eqnarray*}
\chi_{\|\nabla v_n\|_2^2\ge S^{\frac{N}{2}}+o(1)}
=\left\{\aligned
&1,\quad \|\nabla v_n\|_2^2\ge S^{\frac{N}{2}}+o(1),\\
&0,\quad \|\nabla v_n\|_2^2<S^{\frac{N}{2}}+o(1).
\endaligned\right.
\end{eqnarray*}
It follows from the method of the Lagrange multiplier that $u_{a,\mu,-}$ is a solution of the Schr\"{o}dinger equation~\eqref{eqd106} for the Lagrange multiplier $\lambda_{a, \mu, -}<0$.  Since $u_{a,\mu, -}$ is real valued, nonnegative, radially symmetric and radially decreasing, by the maximum principle, we also know that $u_{a,\mu, -}$ is also positive.
\end{proof}

\vskip0.12in

\section{The existence of solutions of $(\ref{eqd106})$ for $\mu=\mu^*_a$}
In the case~$\mu=\mu^*_a$, the degenerate submanifold $\mathcal{P}^0_{a,\mu}$ may not be an emptyset by Proposition~\ref{prop001}.  Thus, to establish the existence of solutions of $(\ref{eqd106})$ for $\mu=\mu^*_a$, we need to get a deeper well understood of the $L^2$-Pohozaev manifold $\mathcal{P}_{a,\mu}$ for $\mu=\mu^*_a$.

\vskip0.12in

\subsection{The structure of $\mathcal{P}_{a,\mu^*_a}$}
We begin with the following property.
\begin{proposition}\label{prop002}
Let $N\ge 3$, $2<q<2+\frac{4}{N}$, $a>0$ and $\mu=\mu^*_{a}$.  Then 
\begin{eqnarray}\label{eqnWu0013}
\mathcal{P}^0_{a,\mu^*_a}=\left\{u\in \mathcal{P}_{a,\mu^*_a}\mid\mu(u)=\mu^*_a\right\}.
\end{eqnarray}
Moreover, any $u\in \mathcal{P}^0_{a,\mu^*_a}$ also satisfies the following equation:
\begin{equation}\label{eqnWu0014}
-2\Delta u-\mu^*_aq\gamma_q|u|^{q-2}u-2^*|u|^{2^*-2}u=\lambda u\quad\text{in }\bbr^N,
\end{equation}
where $\lambda\in\bbr$ is the Lagrange multiplier.
\end{proposition}
\begin{proof}
The conclusion~\eqref{eqnWu0013} follows immediately from $(b)$ of Proposition~\ref{prop001}.  It follows from the definition of the extremal value $\mu^*_a$ given by \eqref{defi001} that $u\in\mathcal{P}^0_{a,\mu^*_a}$ if and only if $u$ is a minimizer of the variational problem~\eqref{defi001}.  Thus, by computing the derivative of the functional $\frac{\left(\|\nabla v\|_2^2\right)^\frac{2^*-q\gamma_q}{2^*-2}}{\|v\|_q^q\left(\|v\|_{2^*}^{2^*}\right)^{\frac{2-q\gamma_q}{2^*-2}}}$ at the minimizer $u$, we derive the equation of $u$ which is given by \eqref{eqnWu0014}.
\end{proof}

\vskip0.12in

As in \cite{AlbuquerqueSilva2020}, we introduce the set $\hat{\mathcal{P}}_{a,\mu}=\left\{u\in \mathcal{S}_a| \mu<\mu(u)\right\}$.  Then by Proposition~\ref{prop001} and $(a)$ of Proposition~\ref{propmup}, it is easy to see that 
\begin{eqnarray*}
\hat{\mathcal{P}}_{a,\mu}=\left\{(u)_s|s>0,u\in \mathcal{P}^+_{a,\mu}\cup \mathcal{P}^-_{a,\mu}\right\}.
\end{eqnarray*}
Let $\overline{\hat{\mathcal{P}}}_{a,\mu^*_a}$ be the closure of $\hat{\mathcal{P}}_{a,\mu^*_a}$ in the $H^1(\bbr^N)$ topology, then by Proposition~\ref{prop001} and $(a)$ of Proposition~\ref{propmup} once more, we immediately have the following.
\begin{lemma}\label{lemWu001}
Let $N\ge 3$, $2<q<2+\frac{4}{N}$ and $a>0$.  Then
\begin{equation*}
\overline{\hat{\mathcal{P}}}_{a,\mu^*_a}=\hat{\mathcal{P}}_{a,\mu_a^*}\cup\left\{(u)_s|s>0, u\in \mathcal{P}_{a,\mu^*_a}^0\right\}.
 \end{equation*}
\end{lemma}

\vskip0.12in

Clearly, by the definitions of $\hat{\mathcal{P}}_{a,\mu}$ and $\overline{\hat{\mathcal{P}}}_{a,\mu^*_a}$ and Proposition~\ref{prop001}, we have $\hat{\mathcal{P}}_{a,\mu}=\mathcal{S}_a$ for $\mu\in (0, \mu^*_a)$ and $\overline{\hat{\mathcal{P}}}_{a,\mu^*_a}=\mathcal{S}_a$.  Moreover, by Lemma~\ref{lemWu001}, we can define two functionals $\tau^\mp_{\mu^*_a}:\overline{\hat{\mathcal{P}}}_{a,\mu^*_a}\to \mathbb{R}$
by
\begin{equation}\label{equ016}
\tau^\mp_{\mu^*_a}(u)=
\begin{cases}t_{\mu^*_a, 2^*}^\mp(u),\quad u\in\hat{\mathcal{P}}_{a,\mu^*_a},
 \\
t_{\mu^*_a, 2^*}^0(u),\quad u\in\overline{\hat{\mathcal{P}}}_{a,\mu^*_a}\setminus \hat{\mathcal{P}}_{a,\mu^*_a}.
\end{cases}
\end{equation}
\begin{lemma}\label{lemWu002}
Let $N\ge 3$, $2<q<2+\frac{4}{N}$ and $a>0$.  Then for any $v\in\mathcal{S}_a$,
\begin{itemize}
    \item [$(a)$]\quad $t^{\mp}_{\mu, 2^*}(v)$ are $C^1$ in terms of $\mu\in (0, \mu^*_a)$ with $t^{+}_{\mu, 2^*}(v)$ increasing and $t^{-}_{\mu, 2^*}(v)$ decreasing.  
   \item [$(b)$]\quad  $\Psi_\mu\left((v)_{t_{\mu, 2^*}^\mp(v)}\right)$ is $C^1$ and decreasing in terms of $\mu\in (0, \mu^*_a)$.
   \item [$(c)$]\quad We have $\lim_{\mu\uparrow \mu^*_a} t_{\mu, 2^*}^\mp(v)=\tau^\mp_{\mu^*_a}(v)$ and 
   \begin{eqnarray*}
   \lim_{\mu\uparrow \mu^*_a} \Psi_\mu\left((v)_{t_{\mu, 2^*}^\mp(v)}\right)=\Psi_{\mu^*_a}\left((v)_{\tau^\mp_{\mu^*_a}(v)}\right),
   \end{eqnarray*}
where $\tau^\mp_{\mu^*_a}(v)$ are given by \eqref{equ016}.
\end{itemize}
\end{lemma}
\begin{proof}
$(a)$\quad The idea of the proof is to apply the implicit function theorem which is similar to that in proving \cite[Lemma~3.2]{WeiWu2022}, so we only sketch it here.  Let
\begin{eqnarray*}
H(\mu,t)=t^2\|\nabla v\|_2^2-\mu\gamma_qt^{q\gamma_q}\|v\|_q^q-t^{2^*}\|v\|^{2^*}_{2^*}.
\end{eqnarray*}
By $(v)_{t^{\mp}_{\mu, 2^*}(v)}\in\mathcal{P}^\mp_{a,\mu}$, it is easy to see that $H\left(\mu,t^{\mp}_{\mu, 2^*}(v)\right)=0$ and $\frac{\partial H(\mu, t^{\mp}_{\mu, 2^*}(v))}{\partial t}\not=0$.  Thus, by the implicit function theorem, $t^{\mp}_{\mu, 2^*}(v)$ are $C^1$ in terms of $\mu\in (0, \mu^*_{a})$ with
\begin{equation*}
\frac{d(t_{\mu, 2^*}^\mp(v))}{d\mu}=\frac{\gamma_qt_{\mu, 2^*}^\mp(v)^{q\gamma_q+1}\|v\|_q^q}{2t_{\mu, 2^*}^\mp(v)^2\|\nabla v\|_2^2-\mu q\gamma_q^2t_{\mu, 2^*}^\mp(v)^{q\gamma_q}\|v\|_q^q-2^*t_{\mu, 2^*}^\mp(v)^{2^*}\|v\|_{2^*}^{2^*}}.
\end{equation*}
In particular, $\frac{\partial t_{\mu, 2^*}^+(v)}{\partial \mu}>0$ and $\frac{\partial t_{\mu, 2^*}^-(v)}{\partial \mu}<0$ for all $\mu\in (0, \mu^*_{a})$.  

\vskip0.06in

$(b)$\quad By the conclusion~$(a)$, $\Psi_{\mu}\left((v)_{t_{\mu, 2^*}^\mp(v)}\right)$ is $C^1$ in terms of $\mu\in (0, \mu^*_{a})$, moreover, by $(v)_{t^{\mp}_{\mu, 2^*}(v)}\in\mathcal{P}^\mp_{a,\mu}$, we also have
\begin{align*}
\frac{d\Psi_{\mu}\left((v)_{t_{\mu, 2^*}^\mp(v)}\right)}{d\mu}&=\frac{\partial}{\partial \mu}\Psi_\mu\left((v)_{t_{\mu, 2^*}^\mp(v)}\right)+\frac{\partial}{\partial t_{\mu, 2^*}^\mp(v)}\Psi_\mu\left((v)_{t_{\mu, 2^*}^\mp(v)}\right)\frac{dt_{\mu, 2^*}^\mp(v)}{d\mu}\\
&=\frac{\partial}{\partial \mu}\Psi_\mu\left((v)_{t_{\mu, 2^*}^\mp(v)}\right)\\
&=-\frac{\left(t_{\mu, 2^*}^\mp(v)\right)^{q\gamma_q}}{q}\|v\|_q^q\\
&<0.
\end{align*}
Thus, $\Psi_{\mu}\left((v)_{t_{\mu, 2^*}^\mp(v)}\right)$ is also decreasing in terms of $\mu\in (0, \mu^*_{a})$.

\vskip0.06in

$(c)$\quad Since $\hat{\mathcal{P}}_{a,\mu}=\mathcal{S}_a$ for $\mu\in (0, \mu^*_a)$, 
for any $u\in\mathcal{S}_a$, by Proposition~\ref{prop001}, there exists $t^{\mp}_{\mu, 2^*}(u)$ such that $t^{+}_{\mu, 2^*}(u)<t^{-}_{\mu, 2^*}(u)$, 
\begin{equation*}
\left\{\array{ll}\left(t^{-}_{\mu, 2^*}(u)\right)^2\|\nabla u\|_2^2-\left(t^{-}_{\mu, 2^*}(u)\right)^{2^*}\|u\|_{2^*}^{2^*}-\mu\gamma_q \left(t^{-}_{\mu, 2^*}(u)\right)^{q\gamma_q}\|u\|_q^q=0,\\
2\left(t^{-}_{\mu, 2^*}(u)\right)^2\|\nabla u\|_2^2-2^*\left(t^{-}_{\mu, 2^*}(u)\right)^{2^*}\|u\|_{2^*}^{2^*}-\mu q\gamma_q^2\left(t^{-}_{\mu, 2^*}(u)\right)^{q\gamma_q}\|u\|_q^q<0,\endarray\right.
\end{equation*}
and
\begin{equation*}
\left\{\array{ll}\left(t^{+}_{\mu, 2^*}(u)\right)^2\|\nabla u\|_2^2-\left(t^{+}_{\mu, 2^*}(u)\right)^{2^*}\|u\|_{2^*}^{2^*}-\mu\gamma_q \left(t^{+}_{\mu, 2^*}(u)\right)^{q\gamma_q}\|u\|_q^q=0,\\
2\left(t^{+}_{\mu, 2^*}(u)\right)^2\|\nabla u\|_2^2-2^*\left(t^{+}_{\mu, 2^*}(u)\right)^{2^*}\|u\|_{2^*}^{2^*}-\mu q\gamma_q^2\left(t^{+}_{\mu, 2^*}(u)\right)^{q\gamma_q}\|u\|_q^q>0.\endarray\right.
\end{equation*}
By the the conclusion~$(a)$, we may assume that 
$t^{\mp}_{\mu, 2^*}(u)\to \overline{t}^{\mp}(u)$ as $\mu\uparrow\mu_a^*$.  Note that we also have $\overline{\hat{\mathcal{P}}}_{a,\mu^*_a}=\mathcal{S}_a$.  Thus, by Lemma~\ref{lemWu001}, either $u\in\hat{\mathcal{P}}_{a,\mu_a^*}$ or $u\in\left\{(v)_s|s>0, v\in \mathcal{P}_{a,\mu^*_a}^0\right\}$.  If $u\in\hat{\mathcal{P}}_{a,\mu_a^*}$ then $\mu^*_a<\mu(u)$.  It follows from Proposition~\ref{prop001} that $\overline{t}^{\mp}(u)=t^{\mp}_{\mu_a^*, 2^*}(u)$.  If $u\in\left\{(v)_s|s>0, v\in \mathcal{P}_{a,\mu^*_a}^0\right\}$ then $\mu^*_a=\mu(u)$,
\begin{equation*}
\left\{\array{ll}\left(\overline{t}^{-}(u)\right)^2\|\nabla u\|_2^2-\left(\overline{t}^{-}(u)\right)^{2^*}\|u\|_{2^*}^{2^*}-\mu^*_a\gamma_q \left(\overline{t}^{-}(u)\right)^{q\gamma_q}\|u\|_q^q=0,\\
2\left(\overline{t}^{-}(u)\right)^2\|\nabla u\|_2^2-2^*\left(\overline{t}^{-}(u)\right)\|u\|_{2^*}^{2^*}-\mu^*_a q\gamma_q^2\left(\overline{t}^{-}(u)\right)^{q\gamma_q}\|u\|_q^q\leq0,\endarray\right.
\end{equation*}
and
\begin{equation*}
\left\{\array{ll}\left(\overline{t}^{+}(u)\right)^2\|\nabla u\|_2^2-\left(\overline{t}^{+}(u)\right)^{2^*}\|u\|_{2^*}^{2^*}-\mu^*_a\gamma_q \left(\overline{t}^{+}(u)\right)^{q\gamma_q}\|u\|_q^q=0,\\
2\left(\overline{t}^{+}(u)\right)^2\|\nabla u\|_2^2-2^*\left(\overline{t}^{+}(u)\right)^{2^*}\|u\|_{2^*}^{2^*}-\mu^*_a q\gamma_q^2\left(\overline{t}^{+}(u)\right)^{q\gamma_q}\|u\|_q^q\geq0.\endarray\right.
\end{equation*}
It follows from $u=(v)_s$ for some $s>0$ and $v\in \mathcal{P}_{a,\mu^*_a}^0$ that $\overline{t}^{+}(u)=\overline{t}^{-}(u)=t^{0}_{\mu_a^*, 2^*}(u)$.  The conclusion of $t_{\mu, 2^*}^\mp(u)$ then follows immediately from the definitions of $\tau^\mp_{\mu^*_a}(u)$ given by \eqref{equ016}.  The conclusions of $\Psi_\mu\left((v)_{t_{\mu, 2^*}^\pm(v)}\right)$ follows from the continuity of $\Psi_\mu\left((v)_{t_{\mu, 2^*}^\pm(v)}\right)$ in terms of $\mu$.
\end{proof}

\vskip0.12in

As in \cite{AlbuquerqueSilva2020}, we also introduce the following two variational problems:
\begin{equation*}
\hat{\Psi}_{\mu^*_a}^\pm=\inf\left\{\Psi_{\mu^*_a}\left(u\right)\mid u\in \mathcal{P}^\pm_{a,\mu^*_a}\cup \mathcal{P}^0_{a,\mu^*_a}\right\}.
\end{equation*}
\begin{proposition}\label{prop007}
Let $N\ge 3$, $2<q<2+\frac{4}{N}$ and $a>0$.  Then 
\begin{itemize}
    \item [$(a)$]\quad $m^\mp(a,\mu)$ are decreasing in terms of $\mu\in(0, \mu^*_a)$ with
\begin{equation}\label{eqnWu0015}
\lim_{\mu\uparrow\mu^*_a}m^\mp(a,\mu)=\hat{\Psi}_{\mu^*_a}^\mp.
\end{equation}
In particular, $\hat{\Psi}_{\mu^*_a}^-\geq \hat{\Psi}_{\mu^*_a}^+$.
    \item [$(b)$]\quad for all $u_b\in \mathcal{P}^\mp_{b,\mu^*_a}$ with $b<a$, we have
\begin{eqnarray}\label{eqnWu0016}
\Psi_{\mu^*_a}(u_b)>\Psi_{\mu^*_a}\left(\left(\sqrt{\frac{a}{b}}u_b\right)_{\tau^\mp_{\mu^*_a}\left(\sqrt{\frac{a}{b}}u_b\right)}\right).
\end{eqnarray}
In particular, $m^\mp(b,\mu^*_a)\ge \hat{\Psi}_{\mu^*_a}^\mp$.
\end{itemize}
\end{proposition}
\begin{proof}
$(a)$\quad By $(b)$ of Lemma~\ref{lemWu002}, it is easy to see that $m^\mp(a,\mu)$ are decreasing in terms of $\mu\in(0, \mu^*_a)$.  It follows from $(c)$ of Lemma~\ref{lemWu002} that $\lim_{\mu\to\mu^*_a}m^\mp(a,\mu)\geq\hat{\Psi}_{\mu^*_a}^\mp$.  On the other hand, for any $\epsilon>0$, we can find $v_{\epsilon}\in\mathcal{P}^\pm_{a,\mu^*_a}\cup \mathcal{P}^0_{a,\mu^*_a}$ such that $\Psi_{\mu^*_a}\left(v_\epsilon\right)<\hat{\Psi}_{\mu^*_a}^\pm+\epsilon$.  By Proposition~\ref{prop001} and $(c)$ of Lemma~\ref{lemWu002}, there exist $t_{\mu, 2^*}^\mp(v_{\epsilon})\to1$ as $\mu\uparrow\mu^*_a$ such that $(v_{\epsilon})_{t_{\mu, 2^*}^\mp(v_{\epsilon})}\in\mathcal{P}^\mp_{a,\mu}$.  It follows that
\begin{equation*}
\lim_{\mu\uparrow\mu^*_a}m^\mp(a,\mu)\leq\lim_{\mu\uparrow\mu^*_a}\Psi_{\mu}\left((v_{\epsilon})_{t_{\mu, 2^*}^\mp(v_{\epsilon})}\right)=\Psi_{\mu^*_a}\left(v_{\epsilon}\right)<\hat{\Psi}_{\mu^*_a}^\mp+\epsilon.
\end{equation*}
By the arbitrariness of $\epsilon>0$, we also have $\lim_{\mu\uparrow\mu^*_a}m^\mp(a,\mu)\leq\hat{\Psi}_{\mu^*_a}^\mp$, which implies that \eqref{eqnWu0015} holds true.  By \eqref{eqnWu0015} and $(a)$ of Lemma~\ref{lem001}, we also have $\hat{\Psi}_{\mu^*_a}^-\geq \hat{\Psi}_{\mu^*_a}^+$.

\vskip0.06in

$(b)$\quad Since $b<a$, $\mathcal{P}^\mp_{b,\mu^*_a}\not=\emptyset$ by Proposition~\ref{prop001} and $(b)$ of Proposition~\ref{propmup}.  Let $u_b\in \mathcal{P}^\mp_{b,\mu^*_a}$ and $\{a_n\}\subset\bbr$ such that $a_n\uparrow a$ as $n\to\infty$, then by $(b)$ of Proposition~\ref{propmup}, $(b)$ of Lemma~\ref{lem001} and $(c)$ of Lemma~\ref{lemWu002}, 
\begin{eqnarray*}
\Psi_{\mu^*_a}(u_b)>\Psi_{\mu^*_a}\left(\left(\sqrt{\frac{a_n}{b}}u_b\right)_{t^\mp_{\mu^*_a}\left(\sqrt{\frac{a_n}{b}}u_b\right)}\right)=\Psi_{\mu^*_a}\left(\left(\sqrt{\frac{a}{b}}u_b\right)_{\tau^\mp_{\mu^*_a}\left(\sqrt{\frac{a}{b}}u_b\right)}\right)+o(1).
\end{eqnarray*}
Thus, \eqref{eqnWu0016} holds true.  By the arbitrariness of $u_b\in \mathcal{P}^\mp_{b,\mu^*_a}$, we also have $m^\pm(b,\mu^*_a)\ge \hat{\Psi}_{\mu^*_a}^\mp$.
\end{proof}

\vskip0.12in

\subsection{The existence of ground-state solutions}
In this section, we shall mainly prove the following.
\begin{proposition}\label{propWu0003}
Let $N\ge 3$, $2<q<2+\frac{4}{N}$ and $a>0$.  Then the variational problem
\begin{equation}\label{eqnWu0018}
\hat{\Psi}_{\mu^*_a}^+=\inf\left\{\Psi_{\mu^*_a}\left(u\right)\mid u\in \mathcal{P}^+_{a,\mu^*_a}\cup \mathcal{P}^0_{a,\mu^*_a}\right\}
\end{equation}
is achieved by some $u_{a,\mu^*_a,+}\in \mathcal{P}^+_{a,\mu^*_a}$, which is real valued, positive, radially symmetric and radially decreasing.
Moreover, $u_{a,\mu^*_a,+}$ also satisfies the Schr\"{o}dinger equation~\eqref{eqd106} for a suitable Lagrange multiplier $\lambda=\lambda_{a, \mu^*_a, +}<0$.
\end{proposition}
\begin{proof}
For the sake of clarity, we divide the proof into two steps.

\vskip0.06in

{\bf Step.~1}\quad We prove that the variational problem~\eqref{eqnWu0018} is achieved by some $u_{a,\mu^*_a,+}$, which is real valued, nonnegative, radially symmetric and radially decreasing.

Indeed, the proof in this step is similar to that of Proposition~\ref{propWu0001}, so we only sketch it.  Let $\mu_n\uparrow\mu^*_a$ as $n\to\infty$ and $u_n$ be the solution of the variational problem~\eqref{eqnWu0003} constructed by Proposition~\ref{propWu0001} for $\mu=\mu_n$, which are real valued, positive, radially symmetric and radially decreasing.  Then by \eqref{eqnWu0017}, $\{u_n\}$ is bounded in $H^1(\bbr^N)$.  Thus, up to a subsequence, $u_n\rightharpoonup u_{a,\mu^*_a,+}$ weakly in $H^1(\bbr^N)$ as $n\to\infty$ for some $u_{a,\mu^*_a,+}\in H^1(\bbr^N)$ and by Strauss's radial lemma (cf. \cite[Lemma~A. I']{BerestyckiLions1983} or \cite[Lemma~1.24]{Willem1996}), we also have $u_n\to u_{a,\mu^*_a,+}$ strongly in $L^r(\bbr^N)$ as $n\to\infty$ for all $2<r<2^*$.  Clearly, $u_{a,\mu^*_a,+}$ is real valued, nonnegative, radially symmetric and radially decreasing.  Recall that $u_n\in \mathcal{P}^+_{a,\mu_n}$ also satisfy the Schr\"{o}dinger equation~\eqref{eqd106} for a suitable Lagrange multiplier $\lambda=\lambda_{n}<0$.
Then 
\begin{align*}
\lambda_n a=(\gamma_q-1)\mu^*_a\|u_n\|_q^q+o(1)
\end{align*}
and
$u_{a,\mu^*_a,+}$ is a weak solution of the following equation:
\begin{equation*}
-\Delta u=\lambda_{a, \mu^*_a, +} u+\mu^*_a|u|^{q-2}u+|u|^{2^*-2}u\quad\text{in }\bbr^N,
\end{equation*}
where $\lambda_{a, \mu^*_a, +}=\lim_{n\to\infty}\lambda_n\leq0$.  It follows that $\lambda_{a, \mu^*_a, +}=0$ if and only if $u_{a,\mu^*_a,+}=0$.  Moreover, since by the standard elliptic regularity theorem, $u_{a,\mu^*_a,+}\in L^{\infty}_{loc}(\bbr^N)$, $u_{a,\mu^*_a,+}$ also satisfies the Pohozaev identity, that is,
\begin{equation*}
\|\nabla u_{a,\mu^*_a,+}\|_2^2=\|u_{a,\mu^*_a,+}\|_{2^*}^{2^*}+\mu^*_a\gamma_q\|u_{a,\mu^*_a,+}\|_q^q.
\end{equation*}
Thus, either $\|v_n\|^{2^*}_{2^*}=\|\nabla v_n\|_2^2=o(1)$ or $\|v_n\|^{2^*}_{2^*}\geq\|\nabla v_n\|_2^2+o(1)\ge S^{\frac{N}{2}}+o(1)$, where $v_n=u_n-u_{a,\mu^*_a,+}$.  Now, if $u_{a, \mu^*_a, +}=0$ then by similar discussions on $\{v_n\}$ as in the case~$(i)$ of the proof of Proposition~\ref{propWu0001} and $(a)$ of Proposition~\ref{prop007}, we have either $\hat{\Psi}_{\mu^*_a}^+=0$ or $\hat{\Psi}_{\mu^*_a}^+\geq\frac{1}{N}S^{\frac{N}{2}}$, which contradicts $(a)$ of Lemma~\ref{lem001} and $(a)$ of Proposition~\ref{prop007}.  Thus, we must have $u_{a, \mu^*_a, +}\not=0$ which also implies that $\lambda_{a, \mu^*_a, +}<0$.  In this case, by the Fatou lemma, we have $\|u_{a, \mu^*_a, +}\|_2^2=a_1\leq a$.  Now, if $a_1<a$ then by similar discussions on $\{v_n\}$ with $v_n=u_n-u_{a,\mu^*_a,+}$ as in the case~$(ii)$ of the proof of Proposition~\ref{propWu0001}, the Brezis-Lieb lemma (cf. \cite[Lemma~1.32]{Willem1996}) and Proposition~\ref{prop007},
\begin{align}\label{eqnWu0021}
\hat{\Psi}_{\mu^*_a}^+=\lim_{n\to\infty}\Psi_{\mu_n}(u_n)\geq \Psi_{\mu^*_a}\left(u_{a,\mu^*_a,+}\right)>\min\left\{\hat{\Psi}_{\mu^*_a}^+, \hat{\Psi}_{\mu^*_a}^-\right\}=\hat{\Psi}_{\mu^*_a}^+,
\end{align}
which is a contradiction.  If $v_n\not\to0$ strongly in $H^1(\bbr^N)$ as $n\to\infty$, then
\begin{align}\label{eqnWu0022}
\hat{\Psi}_{\mu^*_a}^+=\lim_{n\to\infty}\Psi_{\mu_n}(u_n)\geq \Psi_{\mu^*_a}\left(u_{a,\mu^*_a,+}\right)+\frac{1}{N}S^{\frac{N}{2}}\geq\hat{\Psi}_{\mu^*_a}^++\frac{1}{N}S^{\frac{N}{2}},
\end{align}
which is also a contradiction.
Thus, we must have $a_1=a$ and $v_n\to0$ strongly in $H^1(\bbr^N)$ as $n\to\infty$.  It follows that $u_{a,\mu^*_a,+}$ is a solution of the variational problem~\eqref{eqnWu0018}, which is real valued, nonnegative, radially symmetric and radially decreasing.

\vskip0.06in

{\bf Step.~2}\quad We prove that $u_{a,\mu^*_a,+}\in \mathcal{P}^+_{a,\mu^*_a}$ is positive and $u_{a,\mu^*_a,+}$ satisfies the Schr\"{o}dinger equation~\eqref{eqd106} for a suitable Lagrange multiplier $\lambda=\lambda_{a, \mu^*_a, +}<0$.

Indeed, since $u_n\in \mathcal{P}^+_{a,\mu_n}$ and $u_n\to u_{a,\mu^*_a,+}$ strongly in $H^1(\bbr^N)$ as $n\to\infty$, we must have $u_{a,\mu^*_a,+}\in \mathcal{P}^+_{a,\mu^*_a}\cup \mathcal{P}^0_{a,\mu^*_a}$.  Now, suppose the contrary that $u_{a,\mu^*_a,+}\in\mathcal{P}^0_{a,\mu^*_a}$, then by the fact that $u_{a,\mu^*_a,+}$ is a solution of the variational problem~\eqref{eqnWu0018}, we know that $u_{a,\mu^*_a,+}$ is also a solution of the following variational problem:
\begin{equation*}
m^0(a,\mu^*_a)=\inf_{u\in \mathcal{P}_{a, \mu^*_a}^0}\Psi_{\mu^*_a}(u).
\end{equation*}
For any $\varphi\in T_{u_{a,\mu^*_a,+}}\mathcal{S}_a$ where $T_{u_{a,\mu^*_a,+}}\mathcal{S}_a$ is the tangent space of $\mathcal{S}_a$ at $u_{a,\mu^*_a,+}$, by the implicit function theorem, there exists $s(\epsilon)=1-\frac{\|\varphi\|_2^2}{a}\epsilon^2$ for $\epsilon$ sufficiently small such that $s(\epsilon)u_{a,\mu^*_a,+}+\epsilon \varphi\in\mathcal{S}_a$.  We denote $u_\epsilon=s(\epsilon)u_{a,\mu^*_a,+}+\epsilon \varphi$.  By Proposition~\ref{prop001} and the definition of $\tau^+_{\mu^*_a}(u)$ given by \eqref{equ016}, $(u_\epsilon)_{\tau^+_{\mu^*_a}(u_\epsilon)}\in\mathcal{P}^+_{a,\mu^*_a}\cup \mathcal{P}^0_{a,\mu^*_a}$.  Since $u_{a,\mu^*_a,+}\in\mathcal{P}^0_{a,\mu^*_a}$, by the continuity, $\tau^+_{\mu^*_a}(u_\epsilon)\to1$ as $\epsilon\to0$.  Thus, without loss of generality, we may write $\tau^+_{\mu^*_a}(u_\epsilon)=1+\tau(\epsilon)$ where $\tau(\epsilon)\to0$ as $\epsilon\to0$.  Since $(u_\epsilon)_{\tau^+_{\mu^*_a}(u_\epsilon)}\in\mathcal{P}^+_{a,\mu^*_a}\cup \mathcal{P}^0_{a,\mu^*_a}$, we have
\begin{equation*}
\left(\tau^+_{\mu^*_a}(u_\epsilon)\right)^{2-q\gamma_q}\|\nabla u_\epsilon\|_2^2-\left(\tau^+_{\mu^*_a}(u_\epsilon)\right)^{2^*-q\gamma_q}\|u_\epsilon\|_{2^*}^{2^*}-\mu^*_a\gamma_q\|u_\epsilon\|_q^q=0,
\end{equation*}
which, together with $s(\epsilon)=1-\frac{\|\varphi\|_2^2}{a}\epsilon^2$, $\tau^+_{\mu^*_a}(u_\epsilon)=1+\tau(\epsilon)$ and $u_{a,\mu^*_a,+}\in\mathcal{P}^0_{a,\mu^*_a}$, implies that
 \begin{align*}
&-\frac{\|\varphi\|_2^2}{a}\epsilon^2\left(2\|\nabla u_{a,\mu^*_a,+}\|^2_2-2^*\|u_{a,\mu^*_a,+}\|_{2^*}^{2^*}-\mu^*_a q\gamma_q\|u_{a,\mu^*_a,+}\|_q^q\right)+o(\epsilon^2)\nonumber\\
&+\epsilon^2\left(\left\langle-\Delta\varphi-\frac{2^*(2^*-1)}{2}\left|u_{a,\mu^*_a,+}\right|^{2^*-2}\varphi-\mu^*_a\frac{q(q-1)\gamma_q}{2}\left|u_{a,\mu^*_a,+}\right|^{q-2}\varphi,\varphi\right\rangle_{L^2}\right)\nonumber\\
=&\tau^2(\epsilon)\left(\frac{(2^*-q\gamma_q)(2^*-q\gamma_q-1)}{2}\|u_{a,\mu^*_a,+}\|_{2^*}^{2^*}-\frac{(2-q\gamma_q)(1-q\gamma_q)}{2}\|\nabla u_{a,\mu^*_a,+}\|_2^2\right)\nonumber\\
&+\epsilon\tau(\epsilon)\left\langle2^*(2^*-q\gamma_q)|u_{a,\mu^*_a,+}|^{2^*-2}u_{a,\mu^*_a,+}+2(2-q\gamma_q)\Delta u_{a,\mu^*_a,+}, \varphi\right\rangle_{L^2}+o(\tau^2(\epsilon)).
\end{align*}
Since $u_{a,\mu^*_a,+}\in\mathcal{P}^0_{a,\mu^*_a}$, by $2^*>2$ and $2<q<2+\frac{4}{N}$, we have
\begin{eqnarray*}
\frac{(2^*-q\gamma_q)(2^*-q\gamma_q-1)}{2}\|u_{a,\mu^*_a,+}\|_{2^*}^{2^*}-\frac{(2-q\gamma_q)(1-q\gamma_q)}{2}\|\nabla u_{a,\mu^*_a,+}\|_2^2>0.
\end{eqnarray*}
It follows that $\left|\tau(\epsilon)\right|\lesssim\epsilon$ as $\epsilon\to0$ which, together with $u_{a,\mu^*_a,+}\in\mathcal{P}^0_{a,\mu^*_a}$, implies that
\begin{align*}
\Psi_{\mu_a^*}((u_\epsilon)_{\tau^+_{\mu^*_a}(u_\epsilon)})=&\frac{(\tau^+_{\mu^*_a}(u_\epsilon))^2}{2}\|\nabla u_\epsilon\|_2^2-\frac{(\tau^+_{\mu^*_a}(u_\epsilon))^{2^*}}{2^*}\|u_\epsilon\|_{2^*}^{2^*}-\frac{\mu^*_a(\tau^+_{\mu^*_a}(u_\epsilon))^{q\gamma_q}}{q}\|u_\epsilon\|_q^q\nonumber\\
=&\frac{1}{2}\|\nabla u_{a,\mu^*_a,+}\|^2_2-\frac{1}{2^*}\|u_{a,\mu^*_a,+}\|_{2^*}^{2^*}-\frac{\mu^*_a}{q}\|u_{a,\mu^*_a,+}\|^q_q+o(\epsilon)\nonumber\\
&+\epsilon\left\langle-\Delta u_{a,\mu^*_a,+}-\left|u_{a,\mu^*_a,+}\right|^{2^*-2}u_{a,\mu^*_a,+}-\mu^*_a\left|u_{a,\mu^*_a,+}\right|^{q-2}u_{a,\mu^*_a,+},\varphi\right\rangle_{L^2}.
\end{align*}
If there exists $\varphi\in T_{u_{a,\mu^*_a,+}}\mathcal{S}_a$ such that
\begin{eqnarray}\label{eqnWu0019}
\left\langle-\Delta u_{a,\mu^*_a,+}-\left|u_{a,\mu^*_a,+}\right|^{2^*-2}u_{a,\mu^*_a,+}-\mu^*_a\left|u_{a,\mu^*_a,+}\right|^{q-2}u_{a,\mu^*_a,+},\varphi\right\rangle_{L^2}\not=0,
\end{eqnarray}
then we have 
\begin{equation*}
\hat{\Psi}_{\mu^*_a}^+\leq\Psi_{\mu_a^*}((u_\epsilon)_{\tau^+_{\mu^*_a}(u_\epsilon)})<\Psi_{\mu_a^*}(u_{a,\mu^*_a,+})=\hat{\Psi}_{\mu^*_a}^+
\end{equation*}
for $\epsilon$ sufficiently small which is a contradiction, and thus, $u_{a,\mu^*_a,+}\in \mathcal{P}^+_{a,\mu^*_a}$.  It follows from the method of the Lagrange multiplier that $u_{a,\mu^*_a,+}$ satisfies the Schr\"{o}dinger equation~\eqref{eqd106} for a suitable Lagrange multiplier $\lambda=\lambda_{a, \mu^*_a, +}<0$.  Since $u_{a,\mu^*_a,+}$ is nonnegative, by the maximum principle, $u_{a,\mu^*_a,+}$ is positive.  It remains to prove that there exists $\varphi\in T_{u_{a,\mu^*_a,+}}\mathcal{S}_a$ such that \eqref{eqnWu0019} holds true.  Suppose the contrary, then $u_{a,\mu^*_a,+}$ satisfies the Schr\"{o}dinger equation~\eqref{eqd106} for a suitable $\lambda\in\bbr$ in the weak sense.  By the standard elliptic regularity theorem, we know that $u_{a,\mu^*_a,+}$ is also a classical solution of the Schr\"{o}dinger equation~\eqref{eqd106}.  On the other hand, since $u_{a,\mu^*_a,+}\in\mathcal{P}^0_{a,\mu^*_a}$, by Proposition~\ref{prop002}, $u_{a,\mu^*_a,+}$ is also a solution of \eqref{eqnWu0014} for a suitable $\lambda'\in\bbr$.  Again, by the standard elliptic regularity theorem, we know that $u_{a,\mu^*_a,+}$ is also a classical solution of the Schr\"{o}dinger equation~\eqref{eqnWu0014}.  Thus, by \eqref{eqd106} and \eqref{eqnWu0014}, we know that $(2^*-2)|u_{a,\mu^*_a,+}|^{2^*-2}-\mu_a^*(2-q\gamma_q)|u_{a,\mu^*_a,+}|^{q-2}$ is a constant in $\bbr^N$, which contradicts $u_{a,\mu^*_a,+}\in H^1(\bbr^N)$ and $\|u_{a,\mu^*_a,+}\|_2^2=a>0$.  Thus, there exists $\varphi\in T_{u_{a,\mu^*_a,+}}\mathcal{S}_a$ such that \eqref{eqnWu0019} holds true.
\end{proof}

\vskip0.12in

\subsection{The existence of mountain-pass solutions}
In this section, we shall study the variational problem:
\begin{equation}\label{eqnWu0020}
\hat{\Psi}_{\mu^*_a}^-=\inf\left\{\Psi_{\mu^*_a}\left(u\right)\mid u\in \mathcal{P}^-_{a,\mu^*_a}\cup \mathcal{P}^0_{a,\mu^*_a}\right\}.
\end{equation}
We begin with the following estimate of $\hat{\Psi}_{\mu^*_a}^-$.
\begin{lemma}\label{lemWu003}
Let $N\geq3$, $2<q<2+\frac{4}{N}$ and $a>0$.  Then
\begin{eqnarray*}
\hat{\Psi}_{\mu^*_a}^-<\hat{\Psi}_{\mu^*_a}^++\frac{1}{N}S^{\frac{N}{2}}.
\end{eqnarray*}
\end{lemma}
\begin{proof}
Since by Proposition~\ref{propWu0003}, the variational problem~\eqref{eqnWu0018} is achieved by some $u_{a,\mu^*_a,+}\in \mathcal{P}^+_{a,\mu^*_a}$, which is real valued, positive, radially symmetric and radially decreasing, we only need to prove $\hat{\Psi}_{\mu^*_a}^-<m^+(a,\mu^*_a)+\frac{1}{N}S^{\frac{N}{2}}$.
Note that by $u_{a,\mu^*_a,+}\in \mathcal{P}^+_{a,\mu^*_a}$ and Proposition~\ref{prop001}, $0<\mu<\mu\left(u_{a,\mu^*_a,+}\right)$.  Thus, by $(c)$ of Lemma~\ref{lemWu002}, the rest of the proof is the same as that of Lemma~\ref{lem002} with trivial modifications, so we omit it here.
\end{proof}

\vskip0.12in

With the above estimate of $\hat{\Psi}_{\mu^*_a}^-$, we can prove the following.
\begin{proposition}\label{propWu0004}
Let $N\ge 3$, $2<q<2+\frac{4}{N}$ and $a>0$.  Then the variational problem~\eqref{eqnWu0020}
is achieved by some $u_{a,\mu^*_a,-}\in \mathcal{P}^-_{a,\mu^*_a}$, which is real valued, positive, radially symmetric and radially decreasing.
Moreover, $u_{a,\mu^*_a,-}$ also satisfies the Schr\"{o}dinger equation~\eqref{eqd106} for a suitable Lagrange multiplier $\lambda=\lambda_{a, \mu^*_a, -}<0$.
\end{proposition}
\begin{proof}
For the sake of clarity, we also divide the proof into two steps.

\vskip0.06in

{\bf Step.~1}\quad We prove that the variational problem~\eqref{eqnWu0020} is achieved by some $u_{a,\mu^*_a,-}$, which is real valued, nonnegative, radially symmetric and radially decreasing.

Again, the proof in this step is similar to that of Proposition~\ref{propWu0002} and Proposition~\ref{propWu0003}, so we only sketch it.  Let $\mu_n\uparrow\mu^*_a$ as $n\to\infty$ and $u_n$ be the solution of the variational problem~\eqref{eqnWu0004} constructed by Proposition~\ref{propWu0002} for $\mu=\mu_n$, which are real valued, positive, radially symmetric and radially decreasing.  Since $2<q<2+\frac{4}{N}<p<2^*$, by $(a)$ of Proposition~\ref{prop007} and similar estimates in the proof of \cite[Lemma~4.1]{Soave2020-1}, we know that $\{u_n\}$ is bounded in $H^1(\bbr^N)$.  Thus, up to a subsequence, $u_n\rightharpoonup u_{a,\mu^*_a,-}$ weakly in $H^1(\bbr^N)$ as $n\to\infty$ for some $u_{a,\mu^*_a,-}\in H^1(\bbr^N)$ and by Strauss's radial lemma (cf. \cite[Lemma~A. I']{BerestyckiLions1983} or \cite[Lemma~1.24]{Willem1996}), we also have $u_n\to u_{a,\mu^*_a,-}$ strongly in $L^r(\bbr^N)$ as $n\to\infty$ for all $2<r<2^*$.  Clearly, $u_{a,\mu^*_a,-}$ is real valued, nonnegative, radially symmetric and radially decreasing.  By similar arguments in the proof of Proposition~\ref{propWu0003}, we know that
$u_{a,\mu^*_a,-}$ is a weak solution of the following equation
\begin{equation*}
-\Delta u=\lambda_{a, \mu^*_a, -} u+\mu^*_a|u|^{q-2}u+|u|^{2^*-2}u\quad\text{in }\bbr^N,
\end{equation*}
with $\lambda_{a, \mu^*_a, -}=0$ if and only if $u_{a,\mu^*_a,-}=0$, and $u_{a,\mu^*_a,-}$ also satisfies the Pohozaev identity
\begin{equation*}
\|\nabla u_{a,\mu^*_a,-}\|_2^2=\|u_{a,\mu^*_a,-}\|_{2^*}^{2^*}+\mu^*_a\gamma_q\|u_{a,\mu^*_a,-}\|_q^q.
\end{equation*}
Similar to that of Proposition~\ref{propWu0003}, if $u_{a,\mu^*_a,-}=0$ then either $\hat{\Psi}_{\mu^*_a}^-=0$ or $\hat{\Psi}_{\mu^*_a}^-\geq\frac{1}{N}S^{\frac{N}{2}}$.  However, by the definition of $\hat{\Psi}_{\mu^*_a}^-$ and $2<q<2+\frac{4}{N}$, we have
\begin{eqnarray*}
\hat{\Psi}_{\mu^*_a}^-\leq\inf_{u\in \mathcal{P}_{a, \mu}^0}\Psi_\mu(u)=-\frac{2-q\gamma_q}{Nq\gamma_q}\inf_{u\in \mathcal{P}_{a, \mu}^0}\|\nabla u\|_2^2<0.
\end{eqnarray*}
Thus, we must have $u_{a,\mu^*_a,-}\not=0$.  By the Fatou lemma, we have $\|u_{a, \mu^*_a, -}\|_2^2=a_1\leq a$.  Now, by similar estimates of \eqref{eqnWu0021} and \eqref{eqnWu0022} and Lemma~\ref{lemWu003}, we must have $a_1=a$ and $v_n=u_n-u_{a, \mu^*_a, -}\to0$ strongly in $H^1(\bbr^N)$ as $n\to\infty$.  It follows that $u_{a,\mu^*_a,-}$ is a solution of the variational problem~\eqref{eqnWu0020}, which is real valued, nonnegative, radially symmetric and radially decreasing.

\vskip0.06in

{\bf Step.~2}\quad We prove that $u_{a,\mu^*_a,-}\in \mathcal{P}^-_{a,\mu^*_a}$ is positive and $u_{a,\mu^*_a,-}$ satisfies the Schr\"{o}dinger equation~\eqref{eqd106} for a suitable Lagrange multiplier $\lambda=\lambda_{a, \mu^*_a, -}<0$.

The proof in this step is the same as that of Step.~2 of the proof of Proposition~\ref{propWu0003} since that argument is only dependent of the property of the degenerate submanifold $\mathcal{P}^0_{a,\mu^*_a}$, so we omit it here.
\end{proof}

\vskip0.12in

We close this section by the proof of Theorem~\ref{thm001}.

\vskip0.06in

\noindent\textbf{Proof of Theorem~\ref{thm001}:}\quad It follows immediately from Propositions~\ref{propWu0001}, \ref{propWu0002}, \ref{propWu0003} and \ref{propWu0004}.
\hfill$\Box$

\section{Acknowledgments}
The research of this article is supported by NSFC (No. 12171470).

\end{document}